\documentclass[11pt]{article}

\usepackage{amsmath,amssymb,amsthm,mathtools}
\usepackage{tikz,graphicx}
\usetikzlibrary{arrows.meta}
\usetikzlibrary{decorations.markings}
\usepackage{graphicx,subcaption}

\theoremstyle{plain}
\newtheorem*{theorem*}{Theorem}
\newtheorem{theorem}{Theorem}[section] 
\newtheorem{lemma}[theorem]{Lemma}
\newtheorem{proposition}[theorem]{Proposition}

\theoremstyle{definition}

\newtheorem{remark}[theorem]{Remark}
\newtheorem{rhp}{Riemann-Hilbert Problem}[section]
\makeatletter
\@addtoreset{equation}{section}
\makeatother
\numberwithin{equation}{section}

\mathtoolsset{showonlyrefs}

\DeclareMathOperator{\diag}{diag}

\renewcommand{\d}{\,\mathrm{d}}
\renewcommand{\i}{\mathrm{i}}

\DeclareMathOperator{\Tr}{Tr}

\newcommand{\lozr}{
	--++(1,1)--++(0,1)--++(-1,-1)
	--++(0,-1)
}

\newcommand{\lozd}{--++(1,1)--++(1,0)--++(-1,-1)--++(-1,0)
}

\newcommand{\lozu}{--++(1,0)--++(0,1)--++(-1,0)--++(0,-1)
}

\author{Tomas Berggren\footnote{Department of Mathematics, 
	Royal Institute of Technology (KTH),
	Stockholm, Sweden. Email: tobergg@kth.se. Supported by the Swedish Research Council grant (VR) Grant no. 2016-05450 and the G\"oran Gustafsson Foundation.} \and Maurice Duits \footnote{Department of Mathematics, 
		Royal Institute of Technology (KTH),
		Stockholm, Sweden. Email: duits@kth.se. Supported by the Swedish Research Council grant (VR) Grant no. 2016-05450 and the G\"oran Gustafsson Foundation.}}
\title{Correlation functions for determinantal processes defined by infinite block Toeplitz minors}

\date{}

\begin{document}
\maketitle	

\begin{abstract}
	We study the correlation functions for determinantal point processes defined by products of infinite minors of block Toeplitz matrices. The motivation for studying such processes comes from doubly periodically weighted  tilings of planar domains, such as the two-periodic Aztec diamond. Our main results are double integral formulas for the correlation kernels. In general, the integrand is a matrix-valued function built out of a factorization of the matrix-valued weight. In concrete examples the factorization can be worked out in detail and we obtain explicit integrands.  In particular,    we find an alternative proof for a formula  for the two-periodic Aztec diamond  recently derived in \cite{DK}. We strongly believe that also in other concrete cases the double integral formulas are good starting points for asymptotic studies. 
\end{abstract}
	\section{Introduction}

We study the correlation functions for determinantal point processes defined by products of (infinite) minors of \emph{block} Toeplitz matrices. This is a natural extension of processes defined by products of \emph{scalar} Toeplitz minors, which is  very  well-studied in the literature. The infinite point system in the scalar case  covers the Schur process introduced in \cite{O,OR1}, one of the gems of integrable probability that includes  a variety of models such as the longest increasing subsequence of a random permutation,  random tilings of planar domains and random growth models. By  definition, the correlation functions for a determinantal point process can be expressed in terms of determinants of matrices constructed out of a function of two variables, called the correlation kernel. For Schur processes, this correlation kernel can be written in terms of double integral formulas with explicitly known integrands. Saddle point methods are thus at our disposal for the asymptotic analysis of  concrete examples.  This has been a large industry in recent years and we do not attempt to provide a full list of works, but point to \cite{BDS,BoGo,J06,J17} and the references therein, for  introductions to this subject and as general references.
	
	An important motivation for studying the extension to block Toeplitz minors comes from random tilings of planar domains or random dimer configurations with doubly-periodic weights, that was discussed in \cite{KO,KOS,NHB}.  The double periodicity in the weight structure naturally leads in many cases to taking minors of block Toeplitz matrices.    Being a natural and non-trivial extension of the scalar case one may therefore expect a richer structure where new phenomena can be discovered. A  remarkable feature for periodically weighted dimer models is that   a so-called gas region \cite{KOS,NHB} may appear. In such a region the 2-point correlations for the height function decay exponentially with the distance.  However, the integrable structure for these models is relatively unexplored and one reason for this is that the standard techniques for the scalar case are  inadequate.  Explicit double integrals for the kernel, even for $n \to \infty$, are not known generally. In fact, to the best of our knowledge, such a double integral formula is only known in case of the two-periodic Aztec diamond. The first results are by Chhita and Young \cite{CY} and Chhita and Johansson \cite{CJ}, who found a machinery for computing the inverse Kasteleyn matrix explicitly and used that to perform an asymptotic analysis.  See also \cite{BCJ} for further results.

 Of special importance to us is  the recent paper \cite{DK}, where one of us together with Kuijlaars took a different approach for the two-periodic Aztec diamond. Starting from the definition of non-intersecting path ensembles with general block Toeplitz transitions we showed that the correlation kernel can be related to matrix valued orthogonal polynomials. A striking feature for the two-periodic Aztec diamond is that the Riemann-Hilbert problem for the matrix-valued orthogonal polynomials can be solved for finite $n$ explicitly, resulting in a double integral formula where the integrand is a $2 \times 2$ matrix-valued function. This seemingly simpler formula than the one in \cite{CJ} is in particular useful for an asymptotic analysis using classical saddle point methods, as shown in \cite{DK}. However, the Riemann-Hilbert problem analysis of the matrix orthogonal polynomials in \cite{DK} is tailored to the two-periodic Aztec diamond and a generalization to other models is far from obvious. In fact, that the Riemann-Hilbert problem can be solved explicitly for finite $n$ came somewhat as a surprise.  The aim of this paper is to present a systematic approach for deriving such double integral formulas, that will lead to a new proof for the double integral formula of \cite{DK} and  can also be applied to other models. 
 
In Section \ref{sec:finite} we introduce a family of extended determinantal point process  defined by products of minors of block Toeplitz matrices. We will also recall families of symbols that are of special importance to us, and lead to totally non-negative block Toeplitz matrices. The symbols that we are interested in come from the non-intersecting paths on directed weighted graphs where the weights obey a periodicity in the vertical direction. The total non-negativity of the symbols has been discussed before in  a more general context in \cite{LP,LP2}. In Section \ref{sec:finite} we also recall the results in \cite{DK} and state how   the correlation kernel for these processes can be represented in a double integral where the integrand contains the Christoffel-Darboux kernel for certain matrix-valued orthogonal polynomials.

In Section \ref{sec:infinite} we then proceed to the main interest of the present paper: the case of that the minors are of infinite size. Just as in the Schur process, one may hope that the correlation structure of infinite systems has  a simpler integrable structure than the finite systems.  Indeed‚ we will show that such a double integral can be found in case we can find a Wiener-Hopf type factorization (of \eqref{eq:WHfact}) of the matrix-valued orthogonality weight. Given such a factorization the asymptotics for the matrix-valued orthogonal polynomials can be computed using a Riemann-Hilbert formulation.

 The existence of the Wiener-Hopf type factorization is a classical and non-trivial matter. There is a vast amount of literature devoted to such factorization results and we do not attempt do give a full overview here, but only refer to \cite{GGH} for a  survey of results. Since our matrix-valued weights are typically rational functions, existence results and even constructive procedures for such factorizations have been discussed in the literature. In fact, our rational matrices are of a particular type, which makes the picture even clearer. Using the notion of whirls and curls and  their commutation relations \cite{LP}, we recall in Section \ref{sec:fact} a general procedure for obtaining the factorizations we are interested in. This procedure is in general still complicated and it would be very interesting to classify the cases where it can be worked out explicitly or simplified. For instance, in Sections \ref{section:example_aztec_diamond} and \ref{section:example_lozenge_tiling}  we will illustrate our main results by  discussing various examples where this procedure can be fully carried out.  
 
 Section \ref{section:example_aztec_diamond} is devoted to the Aztec diamond. We will present a new proof for the formula for the  two-periodic Aztec diamond as found in \cite{DK} and also present an example with a $3\times 2$-periodic weighting of the Aztec diamond. In Section \ref{section:example_lozenge_tiling} we then discuss $2 \times 2$-periodic lozenge tilings of the hexagon where the vertical side is send to infinity. In all these examples we obtain explicit double integral formulas.  We strongly believe that they are  good starting points for asymptotic studies, as was the case in \cite{DK} for the two-periodic Aztec diamond. We plan to return to  these examples in future work. 
 
\subsection*{Acknowledgments}

We are very grateful to Alexei Borodin, Sunil Chhita, Kurt Johansson and Arno Kuijlaars for many inspiring dicussions. In particular, we thank Alexei Borodin for pointing out the references \cite{LP,LP2} to us and  Sunil Chhita and Christophe Charlier for providing us with codes that we used to simulate the random samples of the domino tilings of the Aztec diamond and the lozenge tilings of the hexagon. 
 
 \section{Products of minors of block Toeplitz matrices} \label{sec:finite}
 
 In this section we introduce determinantal processes defined by products of minors of block Toeplitz matrices with a finite number of points. In Section \ref{sec:infinite} we will take a particular limit to define processes with infinite points. 
 
 \subsection{Products of finite block Toeplitz  minors}
 
 	We start by recalling the definition of $p\times p$ block Toeplitz matrices. Let $\phi$ be a $p\times p$ matrix-valued function that has Fourier series expansion
 	$$
 		\phi(z)=\sum_{k=-\infty} ^\infty \hat \phi(k) z^k, \quad  \hat \phi(k) =\frac{1}{2 \pi \i} \oint_{|z|=1} \phi(z)\frac{\d z}{z^{k+1}}.
 		$$ 
 	Then the infinite block Toeplitz matrix associated to $\phi$ is defined as
 \begin{multline} \label{eq:Tmblocktoeplitz}
	 T_\phi(px+r,py+s)= \left(\hat \phi(y-x)\right)_{r+1,s+1}
 	\quad r,s=0,\ldots,p-1, \quad x,y \in \mathbb Z. 	
 \end{multline}	
 In other words, the matrix $T_\phi$ is a block matrix that is constant along the diagonals,
 $$
 	T_\phi= \begin{pmatrix} \ddots & \ddots& \ddots&\ddots  \\
 		\ddots & \hat \phi_0  &  \hat \phi_1  & \hat \phi_2&\ddots  \\
 		\ddots & \hat \phi_{-1} & \hat \phi_0  &  \hat \phi_1  & \hat \phi_2&\ddots  \\
	 	\ddots & \hat \phi_{-2} & \hat \phi_{-1} &  \hat \phi_0  & \hat \phi_1  & \hat \phi_2&  \ddots  \\
		&\ddots &\ddots & \ddots & \ddots & \ddots& \ddots \\
 \end{pmatrix},
 $$
 and the blocks along each diagonal are given by the Fourier coefficients of the matrix-valued symbol $\phi$. 
 
 In this paper, the entries  of $\phi$ will always be rational functions such that neither $\phi$ nor $\phi^{-1}$ has poles on the unit circle. We also define our matrices so that they have entries in $\mathbb Z \times \mathbb Z$. In the literature, such double infinite Toeplitz matrices are also called Laurent matrices, e.g. \cite{BG}.

 Fix parameters $p,n,N \in \mathbb N$ and consider discrete variables $((m,u_m^j))_{m=1,j=1}^{N-1,pn} \subset \{0,\ldots,N\}\times  \mathbb Z$  taken randomly from the measure
\begin{equation}\label{eq:productofdeterminants}
		\frac{1}{Z_{n,N}}\prod_{m=1}^N \det \left( T_{\phi_m}(u_{m-1}^j,u_{m}^k)\right)_{j,k=1}^{pn},
\end{equation}
where $T_{\phi_m}$ are $p\times p$ block Toeplitz matrices with symbols $\phi_m$.  	The points $u_0^j$ and $u_N^j$ will be fixed and we take them consecutive. More precisely, fix an additional parameter $M \in \mathbb Z$ and set
\begin{equation}\label{eq:endpoints}
		u_0^j=j-1, \quad u_N^j=pM+j-1, \qquad j=1,\ldots,pn.
\end{equation}
We will typically write $u=px+r$ with $0\leq r \leq p-1$.

To ensure that \eqref{eq:productofdeterminants} is indeed a positive measure we will insist that the symbols $\phi_m$ are such that all possible minors of the block Toeplitz matrices are non-negative. In other words, the block Toeplitz matrices are totally non-negative matrices. The classification of totally non-negative Toeplitz matrices is a classical problem. In the scalar case this is given by the  Thoma-Edrei Theorem \cite{Ed,Th} but also  the matrix case  has been discussed in the literature in recent years \cite{LP,LP2}. We return to this later on. 
	
The normalizing constant $Z_{n,N}$ can be computed to be a block Toeplitz minor itself. First note, since we have doubly infinite Toeplitz matrices, that $T_\phi T_\psi= T_{\phi \psi}$ for any two symbols $\phi$ and $\psi$. By the Cauchy-Binet identity, see e.g. \cite{J06}, we therefore find 	
		$$
			Z_{n,N}=  ((np)!)^{N-1} \det \left(T_{\phi} (u_0^j,u_N^k)\right)_{j,k=1}^{pn},
		$$
	where 
 	\begin{equation}\label{eq:defphi}
 		\phi(z)= \prod_{m=1}^N \phi_m(z).
 	\end{equation}
	For the same reasons, the marginal densities for the points at the $m$-th section can be written as the product of two determinants
		$$
			\frac{((np)!)^{N-2}}{Z_{n,N}}  \det \left( T_{\phi_{0,m}}(u_0^j,u_{m}^k)\right)_{j,k=1}^{pn} \cdot  \det \left( T_{\phi_{m,N}}(u_{m}^j,u_{N}^k)\right)_{j,k=1}^{pn}.
		$$
	Here, and from now on, we use the notation 
	\begin{equation} \label{eq:defphikl}
		\phi_{k,\ell}(z)= \prod_{m=k+1}^\ell	\phi_m(z).
	\end{equation}
	Note that the normalizing constant did not change and is independent of $m$.  Processes that can be written as products of two determinants are called biorthogonal ensembles \cite{Bbio} and are a special class of determinantal point processes. The point process defined by \eqref{eq:productofdeterminants} is an extended biorthogonal ensemble. We refer to \cite{B,J06} (and the references therein) for more background on determinantal point processes and their properties.

	\begin{figure}[t]
		\begin{center}
		\begin{tikzpicture}[scale=.6]
	\tikzset{->-/.style={decoration={
				markings,
				mark=at position .5 with {\arrow{stealth}}},postaction={decorate}}}

			\draw[<->] (-2,7)--(-2,6)--(-1,6);
			\draw (-2,7) node[left] {$u$};
				\draw (-1,6) node[below] {$m$};
				
					\draw (1,0) node[below] {$m=0$};
						\draw (5,0) node[below] {$m=N$};
			\draw[dashed,red] (1,0)--(1,8);			
			\draw[dashed,red] (5,0)--(5,8);			
			
	\foreach \x in {0,1,...,5}
	{\draw[->-] (\x,7)--(\x+1,7);
		\foreach \y in {0,1,...,6}
		{\draw[->-] (\x,\y)--(\x+1,\y+1);
			\draw[->-] (\x,\y)--(\x+1,\y);
	}}
	
	\foreach \y in {1,2,3,4}
	{\draw (1,\y) node[circle,fill,inner sep=2.5pt]{};
		\draw (5,\y+2) node[circle,fill,inner sep=2.5pt]{};}

	\foreach \y in {0,1,...,6}
	{\draw (0,\y)--(6,\y);
		\foreach \x in {0,1,...,5}
		{\draw (\x,\y)--(\x+1,\y+1);
	}}
	{\draw (0,7)--(6,7);}
	\foreach \y in {1,2,3,4}
	{\draw (1,\y) node[circle,fill,inner sep=2pt]{};
		\draw (5,\y+2) node[circle,fill,inner sep=2pt]{};}
	\foreach \x/\y in {1/1, 1/2, 2/1, 2/4, 3/6, 4/4, 4/5, 4/6}
	{\draw[line width=1mm] (\x,\y)--(\x+1,\y);
	\draw (\x+1,\y)  node[circle,fill,inner sep=2.5pt]{};	}
	\foreach \x/\y in {1/3, 1/4, 2/2,2/5, 3/1, 3/3, 3/4, 4/2}
	{\draw[line width=1mm] (\x,\y)--(\x+1,\y+1);
		\draw (\x+1,\y+1)  node[circle,fill,inner sep=2.5pt]{};
	}

	\end{tikzpicture}\qquad 
		 \begin{tikzpicture}[scale=0.5]
		 \tikzset{->-/.style={decoration={
		 			markings,
		 			mark=at position .5 with {\arrow{stealth}}},postaction={decorate}}}
	 			
	 				\draw (1,-4.5) node[below] {$m=0$};
	 			\draw (9,-4.5) node[below] {$m=N$};
	 			\draw[dashed,red] (1,-4.5)--(1,5);			
	 			
				\draw[ dashed,red] (1,-4.5)--(1,5);			
	\draw[help lines, dashed,red] (9,-4.5)--(9,5);			
		 \foreach \y in {-4,-3,-2,-1,0,1,2,3,4}
		 {
		 	\foreach \x in {0,1,2,3,4,5,6,7,8,9}
		 	{\draw[->,>=stealth] (\x-.5,\y)--(\x+.5,\y);
		 		\draw (9.5,\y)--(10.5,\y);
		 }}
		 \foreach \x in {0,2,4,6,8,10}
		 {\foreach \y in {-4,-3,-2,-1,0,1,2,3}
		 	{\draw[-<,>=stealth] (\x,\y-.5)--(\x,\y+.5);
		 		\draw (\x,3.5)--(\x,4.5);}}
		 \foreach \y in {-3,-2,-1,0,1,2,3,4}
		 {\foreach \x in {0,2,4,6,8}
		 	\draw [->-](\x+1,\y)to(\x+2,\y-1);}
	
		 \foreach \y in {0,1,2,3}
		 {\draw (1,\y) node[circle,fill,inner sep=2pt]{};
		 	\draw (9,\y-4) node[circle,fill,inner sep=2pt]{};}
	 	{\foreach \x/\y in  {1/3,1/2,1/1,2/3,2/2,2/0,3/3,2/-3,3/-3,4/3,4/1,4/-1,5/-4,6/-4,7/-4,8/-4,6/-3,7/-3,8/-3,5/3,6/3,8/-1,8/-2,7/3,5/1,6/-1 } 
	 				{\draw[line width=.8mm] (\x,\y)--(\x+1,\y);
 					\draw (\x+1,\y) node[circle,fill,inner sep=2.5pt]{};}} 
 				
 			{\foreach \x/\y in  {1/0,3/2,3/0,4/-3,5/-1,7/-1 } 
{ 			\draw[line width=.8mm] (\x,\y)--(\x+1,\y-1);
 			\draw (\x+1,\y-1) node[circle,fill,inner sep=2.5pt]{};}} 
 		
 			{\foreach \x/\y in  {2/1, 2/-1,2/-2,6/-2,8/3,8/2,8/1,8/0,6/1,6/0 } 
 				 			\draw[line width=.8mm] (\x,\y)--(\x,\y-1);
 			}	
	
		 \end{tikzpicture}
		   \caption{Two examples of point configurations $\{(m,u_m^j)\}_{m=0,j=1}^{N,pn}$. At each vertical line we have $pn$   points. The points at $m=0$ and $m=N$ are fixed and at placed at consecutive integers but with a shift of $pM\in p\mathbb Z$ at the end. In these examples, the points lie on non-intersecting paths placed on an underlying directed graph. }
		\end{center}
		
	\end{figure}
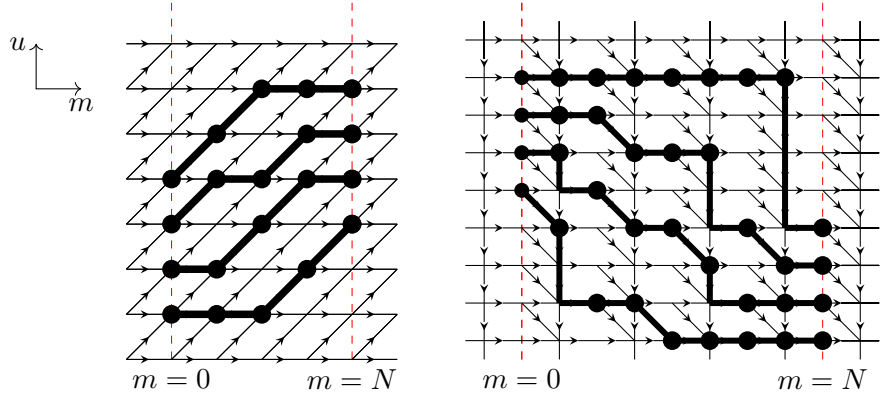
		Point processes defined by \eqref{eq:productofdeterminants} arise in  ensembles of non-intersecting paths. The construction is standard and discussed, for example, in \cite{J06}.  In Sections \ref{section:example_aztec_diamond} and \ref{section:example_lozenge_tiling} we will see two classical examples, related to  domino tilings of the Aztec diamond and lozenge tilings of the hexagon. In both examples, we are given a directed weighted graph with no cycles and with vertices in $\{0,1,\ldots, N\}\times \mathbb Z$. The weights are on the edges. If two vertices are connected by a path, then  we define the weight of that path to be the product of the weights of the edges that form the path. We also assign a weight to a collection of paths by taking the product of the weights of individual paths.   Then we can  consider the space of collections of $np$  paths that  start at $j-1$ and end at $pM+j-1$ for $j=1, \ldots, pn$ and are conditioned never to intersect.    Each such collection of paths has a weight and by normalizing this weight by the sum of all weights of possible collections we obtain a probability measure on the space of non-intersecting paths.  Then the Lindstr\"om-Gessel-Viennot Theorem  \cite{GV,L} says that the vertical positions of the paths, with   $u_m^j$ denoting the vertical position of  the $j$-th path at level $m$, then have a joint probability distribution given by \eqref{eq:productofdeterminants}. The matrices $T_m$ are determined by the weights on the edges of the graph. If these weights have a periodic structure in the vertical direction, then the matrices will indeed be block Toeplitz matrices. We will see explicit examples in the next paragraph.

	\begin{remark}
		Without loss of generality, we can always take the parameter $M=0$. Indeed, one can always add a trivial  last step by taking $T_{N+1}(x,y)=1$ if $y-x=M$ and $0$  otherwise. This is however slightly artificial and in many concrete examples it is more natural to include a general parameter~$M$.   
	\end{remark}

	\subsection{Certain symbols for totally non-negative block Toeplitz matrices}
	
  We will now present an overview of symbols that will lead to totally non-negative Toeplitz matrices, that will be coming from ensembles of non-intersecting paths on directed weighted graphs.  It will be convenient to separate the cases $p=1$, $p=2$ and the general case $p\geq 2$.

	In the scalar case $p=1$  the block matrices reduces to scalar infinite Toeplitz matrices 
\begin{equation} \label{eq:Tmtoeplitz}
T_{\phi_m} (x,y)=\hat \phi_m(y-x)
=\frac{1}{2 \pi \i} \oint_{|z|=1} \phi_m(z) \frac{\d z}{z^{y-x+1}}, \qquad  x,y \in \mathbb Z.
\end{equation}	
This means that the jump probability $T_{\phi_m}(x,y)$ to go from $x$ to $y$ depends  only on the value $y-x$ of the jump  and is driven by the same distribution for all starting points. The symbol $\phi_m$ for the Toeplitz matrix is the generating function for this distribution. Relevant symbols for $p=1$ are
\begin{equation}\label{eq:bernoulliup}
\phi^{b,\uparrow}(z)= 1+a z,
\end{equation}
\begin{equation}\label{eq:bernoullidown}
\phi^{b,\downarrow}(z)= 1+a /z,
\end{equation}
\begin{equation}\label{eq:geometricup}
\phi^{g,\uparrow} (z)= \frac{1}{1-a z},
\end{equation}
and  \begin{equation}\label{eq:geometricdown}
\phi^{g,\downarrow}(z)= \frac{1}{1-a/ z}.
\end{equation}
In \eqref{eq:geometricup} and \eqref{eq:geometricdown} we assume $0<a<1$, but only $a>0$ in \eqref{eq:bernoulliup} and \eqref{eq:bernoullidown}. The transition matrices with symbols \eqref{eq:bernoulliup} and \eqref{eq:bernoullidown} correspond to random  Bernoulli steps and \eqref{eq:geometricup} and \eqref{eq:geometricdown} to geometric steps up and down respectively. 

It can be shown that these choices for $\phi_m$ give rise to totally non-negative Toeplitz matrices. In fact, the Edrei-Thoma Theorem \cite{Ed,Th} says that all totally positive Toeplitz matrices have a symbol that can be written, up to an additional exponential factor,  as a product of terms as in \eqref{eq:bernoulliup}--\eqref{eq:geometricdown}.

\begin{figure}[t]
	\begin{center}
		\begin{tikzpicture}[scale=.9]
	\tikzset{->-/.style={decoration={
				markings,
				mark=at position .5 with {\arrow{stealth}}},postaction={decorate}}}
	\foreach \y in {-4,-3,-2,-1,0,1,2,3,4}
		{\filldraw (0,\y) circle(2pt);
		\filldraw (1,\y) circle(2pt);\draw[->,>=stealth] (0,\y)--(.5,\y);\draw[-] (0,\y)--(1,\y);}
			\foreach \y in {-4,-3,-2,-1,0,1,2,3,4}
			{\draw[->,>=stealth] (0,\y)--(.5,\y+.5);\draw[-] (0,\y)--(1,\y+1);}	\draw (.5,5) node  {{$\vdots$}};
			\draw (.5,-5) node  {{$\vdots$}};

	\draw (0.75,-2) node[above] {\tiny{$b_1$}};
\draw (0.75,-1) node[above] {\tiny{$b_2$}};
\draw (0.75,0) node[above] {\tiny{$b_3$}};
\draw (0,-2.5) node  {\tiny{$a_0b_1$}};
\draw (0,-1.5) node  {\tiny{$a_1b_2$}};
\draw (0,-.5) node {\tiny{$a_2b_3$}};

\draw (0.75,3) node[above] {\tiny{$b_{0}$}};
\draw (0,3.5) node {\tiny{$a_{0}b_1$}};
\draw (0.75,4) node[above] {\tiny{$b_1$}};
\draw (0,2.5) node  {\tiny{$a_{p-1}b_0$}};
\draw (0.75,2) node[above] {\tiny{$b_{p-1}$}};

\draw (0,1.5) node {\tiny{$\vdots$}};
\draw (0,.5) node  {\tiny{$\vdots$}};
	\draw (.5,-6) node  {{$\phi^{b,\uparrow}$}};
			\end{tikzpicture}\quad\quad\quad\quad
				\begin{tikzpicture}[scale=.9]
			\tikzset{->-/.style={decoration={
						markings,
						mark=at position .5 with {\arrow{stealth}}},postaction={decorate}}}
			\foreach \y in {-4,-3,-2,-1,0,1,2,3,4}
			{\filldraw (0,\y) circle(2pt);
				\filldraw (1,\y) circle(2pt);\draw[->,>=stealth] (0,\y)--(.5,\y);\draw[-] (0,\y)--(1,\y);}
			\foreach \y in {-4,-3,-2,-1,0,1,2,3,4}
			{\draw[->,>=stealth] (0,\y)--(.5,\y-.5);\draw[-] (0,\y)--(1,\y-1);}	\draw (.5,5) node  {{$\vdots$}};
			\draw (.5,-5) node  {{$\vdots$}};
				\draw (0.257,-2) node[above] {\tiny{$b_1$}};
			\draw (0.25,-1) node[above] {\tiny{$b_2$}};
			\draw (0.25,0) node[above] {\tiny{$b_3$}};
			\draw (1,-2.5) node  {\tiny{$a_0b_0$}};
			\draw (1,-1.5) node  {\tiny{$a_1b_1$}};
			\draw (1,-.5) node {\tiny{$a_2b_2$}};
			
			\draw (0.25,3) node[above] {\tiny{$b_{0}$}};
			\draw (1,3.5) node {\tiny{$a_{0}b_0$}};
			\draw (0.25,4) node[above] {\tiny{$b_1$}};
			\draw (1,2.5) node  {\tiny{$a_{p-1}b_{p-1}$}};
			\draw (0.25,2) node[above] {\tiny{$b_{p-1}$}};
			
			\draw (1,1.5) node {\tiny{$\vdots$}};
			\draw (1,.5) node  {\tiny{$\vdots$}};
			\draw (.5,-6) node  {{$\phi^{b,\downarrow}$}};
			\end{tikzpicture}\quad\quad
			\begin{tikzpicture}[scale=.9]
			\tikzset{->-/.style={decoration={
						markings,
						mark=at position .5 with {\arrow{stealth}}},postaction={decorate}}}
			\foreach \y in {-4,-3,-2,-1,0,1,2,3,4}
			{\filldraw (0,\y) circle(2pt);
				\filldraw (1,\y) circle(2pt);\draw[->,>=stealth] (0,\y)--(.5,\y);\draw[-] (0,\y)--(1,\y);}
			\foreach \y in {-4,-3,-2,-1,0,1,2,3,4}
			{\draw[->,>=stealth] (0,\y)--(0,\y-0.5);\draw[-] (0,\y)--(0,\y-1);}
			\draw (0.5,4) node[above] {\tiny{$b_1$}};
			\draw (0.5,3) node[above] {\tiny{$b_0$}};
			\draw (0.5,2) node[above] {\tiny{$b_{p-1}$}};
				\draw (0,3.5) node[left]  {\tiny{$a_0$}};
			\draw (0,2.5) node[left]  {\tiny{$a_{p-1}$}};
			\draw (0,1.5) node[left] {\tiny{$a_{p-2}$}};
			
				\draw (0.5,-1) node[above] {\tiny{$b_{1}$}};
			\draw (0,-1.5) node[left]  {\tiny{$a_{0}$}};
			\draw (0.5,-2) node[above] {\tiny{$b_0$}};
			\draw (0,-2.5) node[left]  {\tiny{$a_{p-1}$}};
				\draw (0.5,0) node[above] {\tiny{$b_{2}$}};
			
			\draw (0,-.5) node[left]  {\tiny{$a_1$}};
				\draw (-0.2,.5) node[left]  {\tiny{$\vdots$}};
					\draw (.5,5) node  {{$\vdots$}};
				\draw (.5,-5) node  {{$\vdots$}};
				\draw (.5,-6) node  {{$\phi^{g,\downarrow}$}};
			\end{tikzpicture}\quad\quad
				\begin{tikzpicture}[scale=.9]
			\tikzset{->-/.style={decoration={
						markings,
						mark=at position .5 with {\arrow{stealth}}},postaction={decorate}}}
			\foreach \y in {-4,-3,-2,-1,0,1,2,3,4}
			{\filldraw (0,\y) circle(2pt);
				\filldraw (1,\y) circle(2pt);
				\draw[->,>=stealth] (0,\y)--(.5,\y);\draw[-] (0,\y)--(1,\y);}
			\foreach \y in {-4,-3,-2,-1,0,1,2,3,4}
			{\draw[->,>=stealth] (0,\y)--(0,\y+.5);\draw[-] (0,\y)--(0,\y+1);}	
			
			\draw (.5,5) node  {{$\vdots$}};
			\draw (.5,-5) node  {{$\vdots$}};

				\draw (0.5,-2) node[above] {\tiny{$b_0$}};
			\draw (0.5,-1) node[above] {\tiny{$b_1$}};
			\draw (0.5,0) node[above] {\tiny{$b_{2}$}};
			\draw (0,-2.5) node[left]  {\tiny{$a_{p-1}$}};
			\draw (0,-1.5) node[left]  {\tiny{$a_0$}};
			\draw (0,-.5) node[left] {\tiny{$a_1$}};
				\draw (0.5,2) node[above] {\tiny{$b_{p-1}$}};
			\draw (0.5,3) node[above] {\tiny{$b_{0}$}};
			\draw (0,3.5) node[left]  {\tiny{$a_{0}$}};
			\draw (0.5,4) node[above] {\tiny{$b_1$}};
			\draw (0,2.5) node[left]  {\tiny{$a_{p-1}$}};
			
			\draw (0,1.5) node[left]  {\tiny{$a_{p-2}$}};
			\draw (-0.2,.5) node[left]  {\tiny{$\vdots$}};
				\draw (.5,-6) node  {{$\phi^{g,\uparrow}$}};
			\end{tikzpicture}
			\end{center}
		\caption{The four different directed weighted graphs that correspond to \eqref{eq:matrixjump1}-- \eqref{eq:matrixjump4}.}
		\label{fig:differentstep}
	\end{figure}
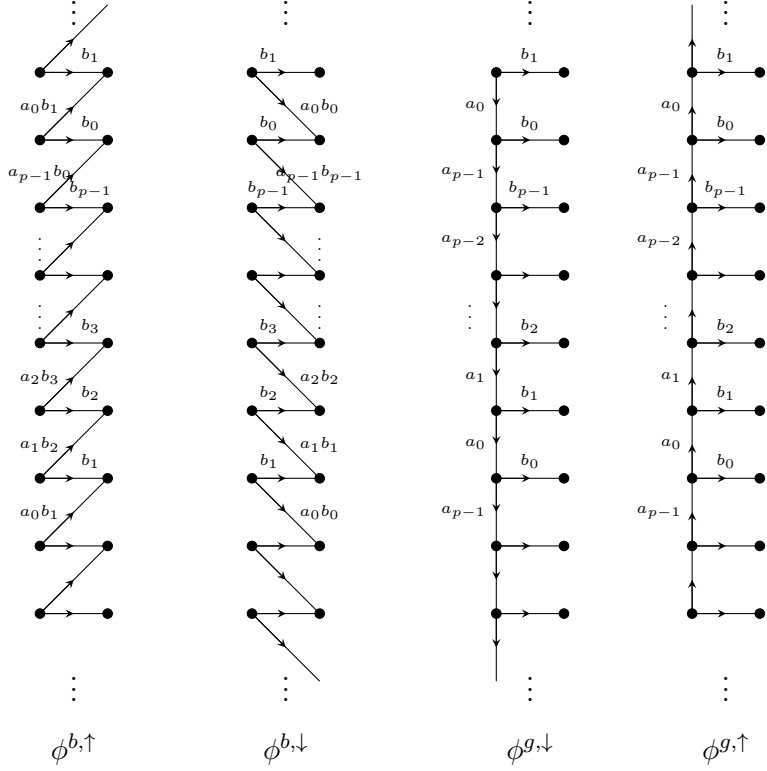

	For $p>1$, we see that the jump not only depends on the value of the jump, but also depends on a modular arithmetic.  Indeed, if $p=2$, we find different distribution for the jump depending on the parity of $x$  and $y$. It turns out that for $p\geq 2$ there are natural analogues of \eqref{eq:bernoulliup} -- \eqref{eq:geometricdown} and it is these analogues that we will be mostly interested in. For $p=2$ they are given by 
	
	\begin{equation}\label{eq:p2bernoulliup}
	\phi^{b,\uparrow}(z)= \begin{pmatrix} 
	b_0 & b_1 a_0\\
	a_1  b_0z & b_1,
	\end{pmatrix}
	\end{equation}
	\begin{equation}\label{eq:p2bernoullidown}
	\phi^{b,\downarrow}(z)= \begin{pmatrix} 
	b_0 & b_1 a_1/z\\
	b_0 a_0&  b_1
	\end{pmatrix},
	\end{equation}
	\begin{equation}\label{eq:p2geometricup}
	\phi^{g,\uparrow} (z)=\frac{1}{1-a_0a_1 z}\begin{pmatrix} 
	b_0 & b_1 a_0\\
	a_1  b_0z & b_1
	\end{pmatrix},
	\end{equation}
	and  \begin{equation}\label{eq:p2geometricdown}
	\phi^{g,\downarrow}(z)= \frac{1}{1-a_0 a_1 /z} \begin{pmatrix} 
	b_0 & a_1 b_1  /z\\
	a_0b_0 &	  b_1
	\end{pmatrix},
	\end{equation}
	where in each case $a_0,a_1,b_0,b_1>0$  and, additionally, $a_0 a_1 <1$ for \eqref{eq:p2geometricup} and  \eqref{eq:p2geometricdown}.

	We now turn to the general case $p\geq 2$.   Let $ {\bf a}=(a_0,\cdots,a_{p-1}) $ be a $p-$tuple of positive parameters and set  $ a = \prod_{i=0}^{p-1}a_i $, let
\begin{equation}\label{eq:defMa}
M(z;{\bf a})= 
\begin{pmatrix}
1 & a_0 & \cdots & 0  \\
0 & 1 & \cdots & 0  \\
\vdots & \vdots & \ddots & \cdots \\
0 & 0 & \cdots & a_{p-2} \\
a_{p-1} z & 0 & \cdots & 1
\end{pmatrix},
\end{equation}
and
\begin{equation}\label{eq:defNa}
N(z;{\bf a})= \frac{1}{1-az}
\begin{pmatrix}
1 & a_0 & \cdots & \prod_{i=0}^{p-2}a_i  \\
\prod_{i=1}^{p-1}a_iz & 1 & \cdots & \prod_{i=1}^{p-2}a_i  \\
\vdots & \vdots & \ddots & \cdots \\
a_{p-2}a_{p-1}z & a_{p-2}a_{p-1}a_0z & \cdots & a_{p-2} \\
a_{p-1}z & a_{p-1}a_0z & \cdots & 1
\end{pmatrix},
\end{equation}
where we assume that $a<1$ in \eqref{eq:defNa}. For a $ p $-tuple of positive parameters $ {\bf b}=(b_0,\cdots,b_{p-1} )$ we set
\begin{equation}
B({\bf b}) = 
\begin{pmatrix}
b_0 & 0 & \cdots & 0 \\
0 & b_1 & \cdots & 0 \\
\vdots & \vdots & \ddots & \cdots \\
0 & 0 & \cdots & b_{p-1}
\end{pmatrix},
\end{equation}

	We say that a $ p\times p $ matrix is the transition matrix for a $ p $-periodic Bernoulli step up respectively down, if it is of the form
	\begin{equation}\label{eq:matrixjump1}
\phi^{b,\uparrow}(z; {\bf a}, {\bf b})=	M(z;{\bf a})B({\bf b}),
	\end{equation}
	respectively
	\begin{equation}\label{eq:matrixjump2}
\phi^{b,\downarrow}(z;{\bf a}, {\bf b})=	M(1/z;{\bf a})^T B({\bf b}).
	\end{equation}
	Similarly we say that a $ p\times p $ matrix is the transition matrix for a $ p $-periodic geometric step up respectively down, if it is of the form
	\begin{equation}\label{eq:matrixjump3}
\phi^{g,\uparrow}(z;{\bf a},{\bf b})=	N(z; {\bf a})B({\bf b}),
	\end{equation}
	respectively
	\begin{equation}\label{eq:matrixjump4}
	\phi^{g,\downarrow}(z;{\bf a},{\bf b})=	N(1/z; {\bf a})^TB({\bf b}).
	\end{equation}
To explain our terminology, we refer to Figure \ref{fig:differentstep}. In that figure, we plotted four directed graphs with weights on the edges. The matrix entry $T_\phi(px+r,py+s)$ where $\phi$ is given by the corresponding expressions \eqref{eq:matrixjump1}, \eqref{eq:matrixjump2}, \eqref{eq:matrixjump3} and \eqref{eq:matrixjump4} then consists of the weight of the unique path connecting $(m,px+r)$ to $(m+1,py+s)$. 

Remarkably, these matrices occur in the characterization of totally non-negative block Toeplitz matrices as discussed in \cite{LP,LP2}. The matrices \eqref{eq:defMa} and \eqref{eq:defNa} are called whirls respectively curls in \cite{LP}. It was proved in \cite{LP} that these symbols lead to total positive block Toeplitz matrices. Thus if each $\phi_m$  is one of these four, the product \eqref{eq:productofdeterminants} defines indeed a probability measure. This is consistent with our picture of non-intersecting paths. In the general situation, we can construct a directed graph on $\{0,\ldots, N\} \times \mathbb Z$ by glueing graphs as indicated in Figure \ref{fig:differentstep}, so that the restriction of the graph to $\{m,m+1\} \times \mathbb Z$ is one of the four types shown. The non-intersecting path model, then leads naturally to the probability measure \eqref{eq:productofdeterminants}.  

\subsection{Doubly periodic models}

As mentioned in the beginning, the fact that we have block Toeplitz matrices, means that there is a periodicity in the vertical direction. We can also obtain a periodicity in the horizontal direction by insisting that for some $q \in \mathbb N$  we have $\phi_{m+q}=\phi_m$   for all $m$. In that case, the model is $q$-periodic in the horizontal direction. The examples of Section \ref{section:example_aztec_diamond} and \ref{section:example_lozenge_tiling}  are of this type.  For such models it is also convenient to replace $N$ by $qN$ such that \eqref{eq:defphi} can be written as 
\begin{equation}\label{eq:phidoubleper}
	\phi(z)= (\Phi(z))^{N}, \qquad \text{with } \quad \Phi(z)= \prod_{m=1}^q \phi_m(z).
	\end{equation}
This structure can have certain advantages when studying the asymptotic behavior, as $N\to \infty$ (which will  not be the focus of this paper). In the scalar case $p=1$, periodic Schur processes have been studied in the literature before \cite{BMRT,M}, by means of steepest descent techniques on the double integral representation of the correlation kernels.

\subsection{Determinantal point processes}

	The Eynard-Mehta Theorem \cite{EM}, see also \cite{B,BR,J06} tells us that  the model \eqref{eq:productofdeterminants} with $T_{\phi_m}$ as in \eqref{eq:Tmblocktoeplitz} defines a determinantal point process.  By definition, this means that there exists a $K_{n}$ such that 
		\begin{equation}\label{eq:Determinantal}
		\mathbb P \left( \text{ points at } (m_1,u_1), \ldots, (m_k,u_k)\right) = \det \left(K_{n}(m_i,u_i;m_j,u_j)\right)_{i,j=1}^k.
		\end{equation}
		One of the key properties of determinantal point processes is that all  information on the point process is thus encoded in the kernel $K_{n}$ and the quantities of interest can be expressed in terms of $K_n$. This will be particularly useful for asymptotic analysis, since we will only need to study the kernel asymptotically as $n \to \infty$. Of course, for this approach to work we need a good control of the kernel for finite  $n$. 
		 
  One of the key observations in \cite{DK} is that the kernel can be expressed in terms of matrix-valued orthogonal polynomials. More precisely, 
\begin{multline}\label{eq:finitekernel}
\left[K_{n}(m,px+i;m',px'+j) \right]_{i,j=0}	^{p-1}=-\frac{\chi_{m>m'}}{2 \pi \i} \oint_{|	z|=1}  \phi_{m',m}(z) 	\frac{\d z}{z^{x-x'+1}} \\
+\frac{1}{(2\pi \i)^2} \oint_{|w|=1} \oint_{|	z|=1} \phi_{m',N}(w) R_n(z,w) 	\phi_{0,m}(z)\frac{w^{x'}}	{z^{x+1}} \frac{\d z \d w}{w^{M+n}},
\end{multline}
where we recall the notation \eqref{eq:defphikl} and $R_n(z,w)$ is  the unique bivariate polynomial of degree $\leq n-1$ in both $z$ and $w$ such that 
$$
\frac{1}{2 \pi \i }\oint_{|w|=1} P(w) \phi(w)  R_n(z,w) \frac{\d w}{w^{M+n}} = P(z),
$$
for every matrix-valued polynomial $P$ of degree $\leq n-1$. In other words, it is the reproducing kernel for the space of matrix-valued polynomials of degree $\leq n-1$ corresponding to the inner product
	$$
		\langle P,Q\rangle=	\frac {1}{2 \pi \i} \oint_{|z|=1} P(z)\phi(z) Q(z)^T \frac{\d z}{z^{M+n}}.
	$$
It is important to note that this a non-hermitian inner product with a complex weight. It is therefore not obvious that $R_n(z,w)$ is well-defined and we refer to \cite[Sec 4]{DK} for a  detailed discussion.

One way of constructing $R_n(z,w)$ more explicitly is the following:
 Start with the matrix $G$ defined as the $pn \times pn$ block matrix with block of size $p\times p$ given by \footnote{Our convention for $G$ and the one in \cite{DK} differ by reordering the columns, i.e. $G_{ij}$ here is $G_{i (n-j+1)}$ in \cite{DK} for $j=1, \ldots,n$.}
	$$
		G=\left(\frac{1}{2  \pi \i}  \oint_{|z|=1} z^{j+i-2}\phi(z) \frac{\d z}{z^{M+n}} \right)_{i,j=1}^n.
	$$ 
Note that  $\det G= (-1)^{n} Z_{n,N}$ and thus $G$ is invertible. Now take a  factorization  $G=LU$ (for now, we do not pose restrictions on $L$ and $U$, but later we will take $L$ and $U$ to be lower and upper triangular respectively). Since $G$ is invertible also $L$   and   $U$ are invertible and we can define
	\begin{equation}  \label{eq:cdkernel1}
		R_n(z,w)=  \sum_{j=0}^{n-1} Q_j(w)^TP_j(z),
	\end{equation}
where 
	$$
		P_{j-1}(z)= \sum_{k=1}^{n} (L^{-1})_{jk} z^{k-1}\text{ and } 
		Q_{j-1}(w)= \sum_{k=1}^n \left((U^{-1})_{kj} \right)^Tw^{k-1},
	$$
where $(L^{-1})_{jk}$  and $(U^{-1})_{kj}$   denote the $jk$-th and $kj$-th block of $L^{-1}$ and $U^{-1}$ respectively. It is easy to verify \cite[Prop. 4.5 ]{DK} that 
	$$
		\frac {1}{2 \pi \i} \oint_{|z|=1} P_j(z)\phi(z) Q_\ell(z)^T \frac{\d z}{z^{M+n}} = \delta_{j,\ell} I_    p. 
	$$
In other words, the $P_j$ and $Q_\ell$ form a biorthogonal family of matrix-valued polynomials with respect to the inner product. As it turns out, the polynomials $P_j$ and $Q_\ell$ depend on the choice of factorization $G=LU$, but the function $R_n(z,w)$ does not and is indeed unique \cite[Lem 4.6 a)]{DK}. 

The above picture simplifies if there exist a factorization $G= LU$ with  $U=H L^T$, where $L$ is a block lower triangular matrix with unit elements on the diagonal and  $H= \diag (H_1,\ldots,H_N)$ is a block diagonal matrix.  Then the triangularity implies that $P_j$ and $Q_\ell$ are polynomials of degree $j$ and $\ell$ respectively.  Moreover, $P_j(z)=p_j(z)$ and $Q_\ell(z)= (H_{\ell+1}^{-1})^T p_{\ell}(z)$   where   $p_j(z)=z^jI_p+\ldots $ is a  matrix-valued polynomial of degree $j$ satisfying 
\begin{equation} \label{eq:orthogonality}
	\frac{1}{2 \pi \i}\oint_{|z|=1} p_j(z) z^k\phi(z)\frac{\d z}{z^{M+n}} =0_p, \qquad k=0,1,\ldots, j-1,
\end{equation}
and $H_{j+1}= \langle p_j,p_j\rangle$. The existence of such polynomials, and/or the existence of the factorization $G=LH L^T$  is not guaranteed. In fact, it may well happen in particular cases that one of the  polynomial for some index $j$ does not exist.  However,  given  the existence, $R_n$ takes the simpler form
\begin{equation}  \label{eq:finitekernel1}
	R_n(z,w)= \sum_{j=0}^{n-1} p_j(w)^T H_{j+1}^{-1} p_{j}(z).
\end{equation}

Although the general existence of $ p_j $ can not be guaranteed, one can  prove \cite[Lem 4.8]{DK} that the   polynomial of the special degree $n$ always exist, based on the fact that $Z_{n,N}>0$.  Moreover, the orthogonal polynomial of that degree  can be characterized by a Riemann-Hilbert problem, similar to the Riemann-Hilbert problem for scalar orthogonal polynomials \cite{FIK}. As we will recall in Section \ref{sec:infinite}, this Riemann-Hilbert problem also characterizes $R_n(z,w)$. We thus arrive at the core of our approach proposed in \cite{DK}. The asymptotic behavior of  Riemann-Hilbert problems for orthogonal polynomials have been studied intensively in the literature after the groundbreaking works \cite{DKMVZ1,DKMVZ2}. One may thus try to apply and extend these techniques to find the asymptotic behavior of $R_n(z,w)$ and use that asymptotic study as a basis to perform a classical steepest descent technique on the double integral formula \eqref{eq:finitekernel}.  The message of this paper is that a relatively simple analysis of the Riemann-Hilbert problem can be performed in case of infinite collections of paths as we will describe in the next section. 

	\section{Infinite systems and their correlation functions}\label{sec:infinite}
In this section we come to the main point of our paper and discuss the model with an infinite collection of paths, $n \to \infty$. Note that by simply taking $n \to \infty$ we run into convergence questions. However, given the fact that the process is determinantal \eqref{eq:Determinantal}, we will instead directly take the limit  $n \to \infty$ in \eqref{eq:finitekernel} while keeping the other parameters fixed. The resulting kernel is the kernel for the determinantal point process with an infinite number of paths.

We will assume that
$$\phi(z)= \prod_{m=1}^N \phi_m(z),$$
and $\phi_m$ are analytic and non-singular in an annulus $\rho < |z|<1/\rho$ containing the unit circle and it has two factorizations
\begin{equation}\label{eq:WHfact}\phi(z)=\phi_+(z)\phi_-(z)= \widetilde \phi_-(z)\widetilde \phi_+(z),\end{equation}
such that 
\begin{enumerate}
	\item $\phi_+^{\pm 1}$, $\widetilde \phi_+^{\pm 1}$ are analytic for $|z|<1$ and continuous for $|z| \leq 1$,
	\item $\phi_-^{\pm 1}$, $\widetilde \phi_-^{\pm 1}$ are analytic for $|z|>1$ and continuous for $|z| \geq 1$,
	\item and $\phi_-, \widetilde \phi_- \sim I_p z^{M}$ as $z \to \infty$. 
\end{enumerate}
In the scalar case $p=1$, the assumption  simplifies since we can always take $\tilde \phi_{\pm}=\phi_\pm$. For $p>1$, this does not necessarily (and most often does not) hold, due to non-commutativity. This makes the case $p>1$ significantly more complicated.

Under the above assumption there are two natural possibilities for obtaining limiting processes. With the non-intersecting path picture in mind, there may be non-trivial limits at the bottom of the pack or at the top. If we fix $m,m',x,x'$ and let $n\to \infty$ then we are looking at the bottom pack, while the top runs off to infinity. On the other hand, if we take 
\begin{equation}\label{eq:shiftininparam}\begin{cases}
x=n+\xi,\\
x'=n+\xi',
\end{cases}
\end{equation}
with $\xi,\xi' \in \mathbb Z$ fixed, then we are focusing at the top pack and sending the bottom part to minus infinity. As it turns out, the factorization $\phi=\phi_+ \phi_-$ is important when we study the bottom part. Similarly, the factorization $\phi=\widetilde \phi_- \widetilde \phi_+$ is important for the top part. 

The following is the first main result of this paper.  We recall that $\phi_{m,m'}$ is  defined in \eqref{eq:defphikl}.

\begin{theorem}\label{thm:main} Consider a model defined by \eqref{eq:productofdeterminants} and  \eqref{eq:endpoints} with weight $\phi$, of the form \eqref{eq:defphi} with $ \phi_m $ analytic and non-singular in an annulus $ \rho < |z| < 1/\rho $, which admits factorizations as in \eqref{eq:WHfact}.
	
	Then, in the limit as $n \to \infty$,
	
	$$ \lim_{n\to \infty}	\left[ K_n(m, px+j; m', px'+i) \right]_{i,j=0}^{p-1}  =
	\left[ K_{bottom}(m, px+j; m', px'+i) \right]_{i,j=0}^{p-1}, $$
	where
	\begin{multline}	\left[ K_{bottom}(m, px+j; m', px'+i) \right]_{i,j=0}^{p-1}
	= - \frac{\chi_{m > m'}}{2\pi \i} \oint_{|z|=1}
	\phi_{m',m}(z) z^{x'-x} \frac{\d z}{z},
	\\
	-  \frac{1}{(2\pi \i)^2}
	\iint_{|z|<|w|} 
	\phi_{m',N}(w) \phi^{-1}_-(w)\phi^{-1}_+(z) \phi_{0,m}(z)  \frac{w^{x'}  \d z \d w }{z^{x+1} (z-w)}\\
	\qquad x,x' \in \mathbb Z, \, 0 < m, m' < N.\end{multline}
	
	In other words, as $n \to \infty$, the bottom paths converge to a determinantal point process with kernel $K_{bottom}$.
	
	In the same way, 	
	\begin{multline*}
	\lim_{n\to \infty}	\left[ K_n(m, p(n+\xi)+j; m', p(n+\xi')+i) \right]_{i,j=0}^{p-1} \\ =
	\left[ K_{top}(m, p\xi+j; m', p\xi'+i) \right]_{i,j=0}^{p-1}  
	\end{multline*}
	where
	\begin{multline}	\left[ K_{top}(m, p\xi +j; m', p \xi'+i) \right]_{i,j=0}^{p-1}
	= - \frac{\chi_{m > m'}}{2\pi \i} \oint_{|z|=1}
	\phi_{m',m}(z) z^{\xi'-\xi} \frac{\d z}{z},
	\\
	+  \frac{1}{(2\pi \i)^2}
	\iint_{|w|<|z|} 
	\phi_{m',N}(w) \widetilde \phi^{-1}_+(w)\widetilde \phi^{-1}_-(z) \phi_{0,m}(z)  \frac{w^{\xi'-M}\d z \d w}{z^{\xi-M+1} (z-w)},  \\
	\qquad \xi,\xi' \in \mathbb Z, \, 0 < m, m' < N.\end{multline}
	In other words, in the limit as $n \to \infty$,  the top paths converge  to a determinantal point process with kernel $K_{top}$.
	
\end{theorem}

The rest of this section is devoted to the proof of this theorem.

As mentioned in Section \ref{sec:finite}, the process defined by \eqref{eq:Tmblocktoeplitz}, \eqref{eq:productofdeterminants}, and \eqref{eq:endpoints}  is determinantal with correlation kernel \eqref{eq:finitekernel}. The matrix-valued orthogonal polynomials and the Christoffel-Darboux kernel $R_n(w,z)$ in \eqref{eq:finitekernel} can be characterized in terms of the following Riemann-Hilbert problem, as mentioned in \cite{DK}.

\begin{rhp} \label{rhp}
	We seek for a $2p\times 2p $ matrix-valued function $Y$ such that 
	\begin{itemize}
		\item $Y$ is analytic in $\mathbb C \setminus\{|z|=1\}$,
		\item we have $$Y_+(z)=Y_-(z) \begin{pmatrix}
		I_p & \frac{\phi(z)}{z^{n+M}} \\
		0 & I_p 
		\end{pmatrix}, \qquad |z|=1,$$
		where $Y_\pm(z)$ denote the limiting values when we approach $z$ from the inside of the unit circle, denoted by $Y_+$, or outside the unit circle, denoted by $Y_-$,
		\item as $z \to \infty$, we have  
		$$Y(z) = (I_{2p}+ \mathcal O(1/z) ) \begin{pmatrix} z^nI_p & 0 \\ 0 & z^{-n} I_p\end{pmatrix}.$$
	\end{itemize}
\end{rhp}

\begin{figure}
	\begin{center}
	\begin{tikzpicture}
		\draw[line width=.8mm] (0,0) circle(2);
		\draw[-{Classical TikZ Rightarrow[scale=1.5]},line width=.8mm] (0,2)--(-.1,2);
			\draw (-.2,2) node[below] {$+$};
				\draw (-.2,2) node[above] {$-$};
					\draw (-3,1) node[above] {$ \begin{pmatrix}
						I_p & \frac{\phi(z)}{z^{n+M}} \\
						0 & I_p 
						\end{pmatrix}$};
			\end{tikzpicture}
			\caption{Jump contours for $Y$.} 
			\label{eq:jumpsforY}
			\end{center}
\end{figure}
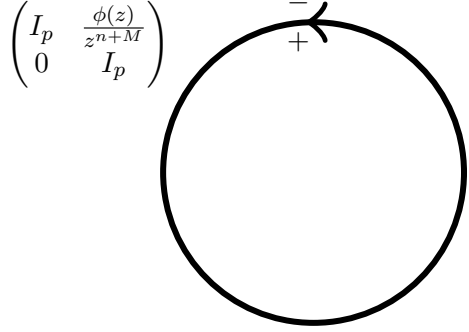

Here, and from now on, we will write $0$ for zero matrices. The dimension of the matrices will be clear from the context. 

If the solution $Y$ to the above Riemann-Hilbert problem exists, then one can show that the upper-left block of the solution is precisely the matrix orthogonal polynomial of degree $n$. Moreover, the Christoffel-Darboux kernel can be expressed in terms of $Y$ as follows,
where $R_n(z,w)$ is given by \cite[Prop. 4.9]{DK},
\begin{equation}\label{eq:CDRHP}
	R_n(z,w)=\frac{1}{z-w}  \begin{pmatrix} 0 & I_p \end{pmatrix} Y^{-1}(w) Y(z) \begin{pmatrix}
I_p\\
0\end{pmatrix}.
\end{equation}
This representation of the Christoffel-Darboux kernel in the context of matrix-valued orthogonal polynomials was first derived by Delvaux \cite[Prop 1.10]{D}.

\begin{proposition}\label{thm:RHP}
	Assume that $\phi(z)$ admits  factorizations as in \eqref{eq:WHfact} and $ \phi^{-1} $ is analytic in an annulus $\rho<1 <1/\rho$  for some $\rho$. Then, as $n \to \infty$ and for $|z|<1$, we have
	$$
	Y(z)=(I+\mathcal O(r^n))
	\begin{pmatrix}
	0 & 	\widetilde \phi_+(z)  \\ 
	-\phi_+^{-1}(z)& 0
	\end{pmatrix},$$
	for any $r$ such that $\max(|z|,\rho)<r<1$.
	
	For $|z|>1$, we have, as $n \to \infty$,
	$$
	Y(z)=(I+\mathcal O(r^n))
	\begin{pmatrix}
	z^{n+M} \widetilde \phi^{-1}_-(z) & 0 \\ 
	0 &z^{-n-M} \phi_-(z)
	\end{pmatrix},$$
	for any $r$ such that $\min(|z|,1/\rho)>1/r$. 
\end{proposition}
\begin{proof}
	The proof follows by a steepest descent analysis of the Riemann-Hilbert problem. Note that for the scalar case, the analysis is rather standard. An excellent introduction to  Riemann-Hilbert problems and their use in asymptotic analysis can be found in \cite{DeiftIntegrable, Deift}. In the matrix case, there are some small but important differences due to the non-commutativity of the factorization \eqref{eq:WHfact}.  
	
	{\bf Step 1}
	The first step is a normalization at infinity. That is, we define 
	$$X(z) = 
	\begin{cases}
	Y(z) \begin{pmatrix} z^{-n} I_p & 0 \\ 0 & z^{n}I_p\end{pmatrix}, & |z|>1,\\
	Y(z), & |z|<1. 
	\end{cases}.$$
	Then it is easy to verify that $X$ satisfies a Riemann Hilbert problem with
	$$X_+(z)=X_-(z) \begin{pmatrix}
	z^{n} I_p & \phi(z)z^{-M} \\ 
	0 & z^{-n} I_p	\end{pmatrix}, \quad |z|=1,$$
	and $$X(z)= I_{2p} +\mathcal O(1/z), \qquad z \to \infty.$$
	The benefit with working with $X$ is that the asymptotics at infinity is normalized and it does not depend on $n$. We also see that the jump matrix is highly oscillating if $n$ is large. In the next step we perform the standard trick of opening of the lenses, to replace the oscillatory jump matrices with exponentially decaying ones. 
	
	\begin{figure}
		\begin{center}
			\begin{tikzpicture}
			\draw[line width=.8mm] (0,0) circle(2);
			\draw[-{Classical TikZ Rightarrow[scale=1.5]},line width=.8mm] (0,2)--(-.1,2);
			\draw (-.2,2) node[below] {$+$};
			\draw (-.2,2) node[above] {$-$};
			
				\draw[line width=.8mm] (0,0) circle(1.25);
			\draw[-{Classical TikZ Rightarrow[scale=1.5]},line width=.8mm] (0,1.25)--(-.1,1.25);
			\draw (-.4,1.1) node[below] {$+$};
			\draw (-.6,1.1) node[above] {$-$};
			
				\draw[line width=.8mm] (0,0) circle(2.75);
			\draw[-{Classical TikZ Rightarrow[scale=1.5]},line width=.8mm] (0,2.75)--(-.1,2.75);
			\draw (.5,2.7) node[below] {$+$};
			\draw (.7,2.7) node[above] {$-$};
			
			\draw (-6,0) node[above] { 
			$ \begin{pmatrix} 0 & \phi(z)z^{-M} \\ -\phi(z)^{-1}z^M  \end{pmatrix}$};
			\draw (-6,1.5) node[above] {  $\begin{pmatrix} I_p & 0\\ z^{-n+M} \phi^{-1}(z)  & I_p  \end{pmatrix}$};
			\draw (-6,-1.5) node[above] { 
			$ \begin{pmatrix} I_p & 0\\  z^{n+M} \phi^{-1}(z)  & I_p  \end{pmatrix} $};
		
		\draw[dashed,help lines,->] (-4,2) .. controls (-3,2.5) and (-3,2.5) .. (-2,2); 
			\draw[dashed,help lines,->] (-3.5,.5)--(-2,.5);
				\draw[dashed,help lines,->] (-4,-1) .. controls (-3,-1.5) and (-3,-1.5) .. (-1,-1); 
		
			\end{tikzpicture}
			\caption{Jump contours for $T$. The jumps on the inner and outer circle give only exponentially small contributions and will be negligible as $n \to \infty$. } 
			\label{eq:jumpsforT}
		\end{center}
	\end{figure}
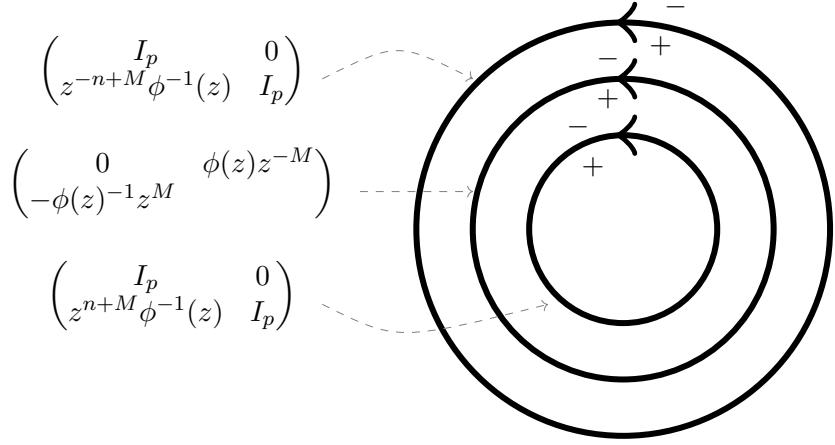

	{\bf  Step 2} We define a new function $T$   out of  $X$   in the following way 
	$$T(z) = 
	\begin{cases}
	X(z) \begin{pmatrix} I_p & 0\\  z^{-n+M} \phi^{-1}(z)  & I_p  \end{pmatrix}, & 1<|z|<1/r,\\
	X(z) \begin{pmatrix} I_p & 0\\ -z^{n+M} \phi^{-1}(z)  & I_p  \end{pmatrix}, & r<|z|<1,\\
	X(z), & |z|>1/r \textrm{ or } |z|<r. 
	\end{cases}$$
	Here we take $\rho<r<1$ so that the annulus defined by $r<|z|<1/r$ is inside the annulus $\rho< 1 <1/\rho$ of analyticity of $\phi^{-1}$. A straightforward check shows that $T$ satisfies a Riemann-Hilbert problem with jump conditions given by 
	$$T_+(z)=T_-(z)  \begin{pmatrix} I_p & 0\\ z^{-n+M} \phi^{-1}(z)  & I_p  \end{pmatrix} , \quad |z|=1/r,$$
	$$T_+(z)=T_-(z)  \begin{pmatrix} I_p & 0\\  z^{n+M} \phi^{-1}(z)  & I_p  \end{pmatrix} , \quad |z|=r,$$
	$$T_+(z)=T_-(z)  \begin{pmatrix} 0 & \phi(z)z^{-M} \\ -\phi(z)^{-1}z^M  \end{pmatrix} , \quad |z|=1,$$
	and asymptotic condition $T(z)=I_{2p}+ \mathcal O(1/z)$ as $z \to \infty$. Note that the jumps on $|z|=r,1/r$ are exponentially decaying as $n\to \infty$, so we expect the contribution from these jumps to be negligible. 
	
	{\bf Step 3} If we ignore the exponentially decaying jumps then we obtain the Riemann-Hilbert problem 
	$$\begin{cases}
	P_+(z)=P_-(z)  \begin{pmatrix} 0 & \phi(z)z^{- M}  \\ -\phi(z)^{-1}z^M &0  \end{pmatrix} ,& |z|=1,\\
	P(z)= I_{2p} + \mathcal O(1/z), & z \to \infty.	\end{cases}
	$$
	The solution to this Riemann-Hilbert problem is easily  constructed and given by 
	$$P(z)= \begin{cases}
	\begin{pmatrix}
	\widetilde \phi^{-1}_-(z)z^M & 0 \\ 
	0 & \phi_-(z)z^{-M}
	\end{pmatrix}, &  |z|>1,\\
	\begin{pmatrix}
	0 & 	\widetilde \phi_+(z)  \\ 
	-\phi_+^{-1}(z)& 0
	\end{pmatrix}, &  |z|<1.\\
	\end{cases}$$
	Note that it is at this point that we need the factorization of $\phi$.

	{\bf Step 4}  It remains to verify that $T$ is close to $P$ for large $n$. Define the function $R$ by 
	$$R(z)= T(z)P(z)^{-1}.$$
	Then $R$ has only jumps for $|z|=r$ and $|z|=1/r$, which are exponentially small. Standard arguments (see for example \cite{DeiftIntegrable, Deift}) now imply that $R$ can be solved in terms of a Neumann series and $$R(z)= I_{2p}+\mathcal O(r^{n}), \qquad n\to \infty,$$
	uniformly on compact subsets of $\mathbb C \setminus \partial \mathbb D$. 
	
	{\bf Step 5} Finally, we trace back the transformations and obtain the behavior of $Y$ in the stated regions. For $|z|<1$ we choose $r$ such that $\max(|z|, \rho)<r<1$ 	and obtain 
	$$Y(z)= X(z)= T(z)=(I+\mathcal O(r^n)) P(z)=(I+\mathcal O(r^n))	\begin{pmatrix}
	0 & 	\widetilde \phi_+(z)  \\ 
	-\phi_+^{-1}(z)& 0
	\end{pmatrix}.$$
	For $|z|>1$ we choose  $1/r<\min (|z|,\rho^{-1})$ 	and now obtain 
	\begin{multline}Y(z)= X(z) \begin{pmatrix} z^{n} I_p & 0 \\ 0 & z^{-n}I_p\end{pmatrix}= T(z)  \begin{pmatrix} z^{n} I_p & 0 \\ 0 & z^{-n}I_p\end{pmatrix}\\=(I+\mathcal O(r^n)) P(z) \begin{pmatrix} z^{n} I_p & 0 \\ 0 & z^{-n}I_p\end{pmatrix}
	=(I+\mathcal O(r^n))\begin{pmatrix}
	z^{n+M} \widetilde \phi^{-1}_-(z) & 0 \\ 
	0 &z^{-n-M} \phi_-(z)
	\end{pmatrix}.\end{multline}
	This proves the statement.
\end{proof}

Now that we have computed the asymptotic behavior of the solution to the Riemann-Hilbert problem we are ready for the proof of Theorem \ref{thm:main}.

\begin{proof}[Proof of Theorem \ref{thm:main}]
	We start with the bottom part of the line ensemble. In \eqref{eq:finitekernel} we deform the contour for $w$ to a circle with a radius slightly bigger than one, but still inside the annulus of analyticity of $\phi$. The contour for $z$ is deformed to a circle with radius less than one. 
	
	Then $$w^{-n} \begin{pmatrix} 0 & I_p \end{pmatrix} Y^{-1}(w) = w^M\phi_-^{-1}(w)\begin{pmatrix} 0 & I_p \end{pmatrix}(I_p+\mathcal O(r^n)), \qquad |w|>1,$$
	and $$Y(z) \begin{pmatrix} I_p \\0 \end{pmatrix}=  -(I_p+\mathcal O(r^n))\phi_+^{-1}(z)\begin{pmatrix} 0 \\  I_p \end{pmatrix}, \quad |z|<1,$$
	where $ r\in(0,1) $ is such that $ |z|,1/|w|<r $. Inserting this into the kernel \eqref{eq:finitekernel} and \eqref{eq:CDRHP}, and then taking the limit $n \to \infty$ gives the statement.
	
	Next we deal with the top part of the line ensemble. We set $x=\xi+n$ and $x'=\xi' +n$ in \eqref{eq:finitekernel}. Then we deform the contour for $w$ to a circle with a radius slightly less  than one  and the contour for $z$ will deformed to a circle with radius slightly bigger than one. 
	Then $$ \begin{pmatrix} 0 & I_p \end{pmatrix} Y^{-1}(w) = \widetilde \phi_+^{-1}(w)\begin{pmatrix}  I_p &0 \end{pmatrix}(I_p+\mathcal O(r^n)), \qquad |w|<1,$$
	and $$z^{-n} Y(z) \begin{pmatrix} I_p \\0 \end{pmatrix}=  z^M (I_p+\mathcal O(r^n))\widetilde \phi_-^{-1}(z)\begin{pmatrix}  I_p \\ 0 \end{pmatrix}, \quad |z|>1,$$
	where $ r\in(0,1) $ is such that $ |w|,1/|z|<r $. Inserting this into the kernel \eqref{eq:finitekernel} and \eqref{eq:CDRHP}, and then taking the limit $n \to \infty$ gives the statement. \end{proof}

	\section{Matrix factorizations} \label{sec:fact}
In light of Theorem \ref{thm:main} it  is pertinent to understand which weights $\phi$ have the factorizations \eqref{eq:WHfact} and how to compute them. This is a classical problem that has been studied intensively and we refer to \cite{GGH} for an overview of results. Under certain conditions one can show that such factorizations exists. Existence, however, is not enough for our purposes, as we are after explicit constructions. We will employ certain commutation relations for the matrix-valued symbols in \eqref{eq:matrixjump1}--\eqref{eq:matrixjump4} that will prove existence by a constructive procedure. In general this procedure is still elaborate, but as we will see in Sections \ref{section:example_aztec_diamond} and \ref{section:example_lozenge_tiling}, it can be worked out explicitly in certain examples.

	\subsection{The case $p=1$}\label{sec:factorization_p1}
	
	For $p=1$ this problem is rather straightforward, especially if we assume the symbols to be of the form \eqref{eq:bernoulliup}--\eqref{eq:geometricdown}. In that case, the factorizations can be found by dividing the symbols into two groups, one with poles and zeros inside the circle and the other with poles and zeros outside the circle. Indeed, if $\phi$ is a rational function on $\mathbb C$ with no zeros or poles on the circle, the necessary and sufficient condition to have a factorization \eqref{eq:WHfact} is that the winding number of $\phi$ is equal to the shift $M$, 
		\begin{equation}\label{eq:windingnumber}
			 \frac{1}{ 2 \pi \i} \oint \frac{\phi'(z)}{\phi(z)} \d z=M.
		\end{equation}
In other words, the number of zeros minus the number poles inside (both counted with respect to multiplicity) the unit disk equals $M.$

This pretty much settles the situation in case each $\phi_m$ is given by one of \eqref{eq:bernoulliup}--\eqref{eq:geometricdown}.  Note that each of these symbols has exactly one pole and one zero. The winding number of each  symbol corresponding to a geometric step up \eqref{eq:geometricup} is zero, since each such term has a pole and zero outside the unit disk. Similarly, symbols corresponding to a geometric jump down have a pole inside the unit disk and a zero at zero. The Bernoulli step up \eqref{eq:bernoulliup} has a pole at infinity but the zero can be both inside and outside the disk. If it is inside, the winding number is $+1$, otherwise it is zero.  The Bernoulli step down \eqref{eq:bernoullidown} has a pole at zero, but the zero can be both inside and outside the disk.  If it is outside, the winding number is $-1$, otherwise it is zero. By setting 
	$$
		M_1= \# \{m \ | \  \phi_m(z)= \phi^{b, \uparrow}(z;a_m) \text{ for some } a_m>1\},
	$$
	and 
	$$
		M_2= \# \{m \ | \ \phi_m(z)= \phi^{b, \downarrow}(z;a_m) \text{ for some } a_m<1\},
	$$
	we thus see that if \eqref{eq:windingnumber} holds then $M=M_1-M_2$. Therefore,  we  have a factorization as in \eqref{eq:WHfact} by setting
$$
	\phi_-(z)=  z^{-M_2} \prod_{m \in I} \phi_m(z), \qquad \phi_+(z) =z^{M_2} \prod_{ m \in I^c} \phi_m(z),
	$$
where 
$$
	I=\left \{ m \mid \phi_m \text{ has a pole or zero in   } \mathbb D \setminus   \{0 \}\right\},
$$
and $I^c= \{1,\ldots, N\}\setminus I$.
This means that  	\begin{multline} K_{bottom}(m, x;m',x')
= - \frac{\chi_{m > m'}}{2\pi \i} \oint_{\gamma}
\phi_{m',m}(z) z^{x'-x} \frac{\d z}{z}
\\
-  \frac{1}{(2\pi \i)^2}
\iint_{|z|<|w|} 
\frac{\prod_{k\in I^c, k > m'} \phi_{k}(w)}{\prod_{k \in I^c, k  > m}  \phi_{k}(z)}\frac{\prod_{\ell \in I, \ell  \leq m} \phi_{\ell}(z)}{ \prod_{\ell \in I, \ell \leq m'} \phi_{\ell}(w)} \frac{w^{x'+M_2}  \d z \d w }{z^{x+M_2+1} (z-w)}\\
\qquad x,x' \in \mathbb Z, \, 0 < m, m' < N.
\end{multline}
This is a well-known formula from Schur processes, see for example \cite[Th. 2.7]{J17}.  Similarly, 
				\begin{multline} K_{top}(m, \xi;m',\xi')
			= - \frac{\chi_{m > m'}}{2\pi \i} \oint_{\gamma}
			\phi_{m',m}(z) z^{\xi'-\xi} \frac{\d z}{z}
			\\
			+ \frac{1}{(2\pi \i)^2}
			\iint_{|z|>|w|} 
			\frac{\prod_{k\in I, k  > m'} \phi_{k}(w)}{\prod_{k \in I, k  > m}  \phi_{k}(z)}\frac{\prod_{\ell \in I^c, \ell  \leq m} \phi_{\ell}(z)}{ \prod_{\ell \in I^c, \ell \leq m'} \phi_{\ell}(w)} \frac{w^{\xi'+M_1}  \d z \d w }{z^{\xi+M_1+1} (z-w)}\\
			\qquad \xi,\xi' \in \mathbb Z, \, 0 < m, m' < N
			.\end{multline}
			This settles the case $p=1$.

	\subsection{Switching rules}\label{sec:switching}
	In case $p>1$ the existence of the factorization is more complicated than the scalar case $p=1$. The reason is that the matrices not necessarily (and most often do not) commute so that we can not simply reorganize the product at the left hand side of \eqref{eq:WHfact} as we could in case $p=1$. However, instead of simply commuting the factors, one can use  certain rules for switching the order of the matrices. 
	
 Suppose we have a factor 
 $\phi_{m}(z)\phi_{m'}(z)$
 and $\phi_{m}$ and $\phi_{m'}$ are regular inside and outside the disk respectively. We would like to switch them, but, of course,  in general we have $\phi_{m}(z)\phi_{m'}(z)\neq \phi_{m'}(z)\phi_{m}(z)$. However,  it is possible to find new $\phi_{m}'(z)$ and $\phi_{m'}'(z)$ that are of the same type (e.g. if $\phi_m$ is a Bernoulli step up, then so is $\phi_m'$ etcetera) such that 
 $$\phi_{m}(z)\phi_{m'}(z)= \phi'_{m'}(z)\phi_{m}'(z),$$
 and $\det \phi_{m}(z)= \det \phi_{m}'(z)$ and $\det \phi_{m'}(z)= \det \phi'_{m'}(z)$. By iterating this process we can find the factorizations in \eqref{eq:WHfact}. 
 
 Let us illustrate the switching rule first for $p=2$.  We recall that for $p=2$ the matrices that we are interested in, are given by \eqref{eq:p2bernoulliup}--\eqref{eq:p2geometricdown}. Then the switching rules are given by the identities in the following lemma. 
	\begin{lemma} The following identities hold
	\begin{equation}
	 \begin{pmatrix}
	 a & b/z\\
	 c & d
	\end{pmatrix}\begin{pmatrix}
	\alpha & \beta \\
	\gamma z & \delta 	 
	\end{pmatrix}=\begin{pmatrix}
	\delta x & \beta \\
	\gamma z &\alpha/x
	\end{pmatrix}\begin{pmatrix}
	d& b /zx\\
	xc & a 	 
	\end{pmatrix}, \quad x=\frac{a \alpha+ b \gamma}{d \delta+\beta c},
	\end{equation}
	\begin{equation}\label{eq:2p_switching_rule_bernoulli}
	\begin{pmatrix}
	a& b\\
	cz & d 	 
	\end{pmatrix}\begin{pmatrix}
	\alpha & \beta \\
	\gamma z & \delta 	 
	\end{pmatrix}=\begin{pmatrix}
	\alpha & \gamma x  \\
	\beta z/x &\delta
	\end{pmatrix}\begin{pmatrix}
	a& c x \\
	b z/x & d	 
	\end{pmatrix}, \quad x=\frac{a \beta + b \delta}{c \alpha+d \gamma},
	\end{equation}
	and
	\begin{equation}\label{eq:2p_switching_rule}
	\begin{pmatrix}
	a& b /z\\
	c & d 	 
	\end{pmatrix}\begin{pmatrix}
	\alpha & \beta /z \\
	\gamma & \delta 	 
	\end{pmatrix}=\begin{pmatrix}
	\alpha & \gamma x /z \\
	\beta/x &\delta
	\end{pmatrix}\begin{pmatrix}
	a& c x /z\\
	b/x & d	 
	\end{pmatrix}, \quad x=\frac{a \beta + b \delta}{c \alpha+d \gamma}.
	\end{equation}
	\end{lemma}
\begin{proof}
	The identities follow easily by a direct verification. 
\end{proof}

For $p>2$ we have similar switching rules, that have already been discussed in the literature. Here we will  follow the definitions and notations as in  \cite{LP}. We refer to that paper and the references therein for more background.  

We start with the map $\eta$ defined by  
\begin{equation}
\eta({\bf a}, {\bf b}) = ({\bf b'},{\bf a'}),
\end{equation}
where ${\bf b'}=(b_0',\ldots,b_{p-1}')$ and   ${\bf a'}=(a_0',\ldots,a_{p-1}')$ are defined by the rules 
\begin{equation}
b'_i = b_{i+1}\frac{k_{i+1}}{k_i}, \qquad a'_i = a_{i-1}\frac{k_{i-1}}{k_i},
\end{equation}
and
\begin{equation}
k_i = \sum_{j=i}^{i+p-1}\prod_{\ell=i+1}^jb_\ell \prod_{\ell=j+1}^{i+p-1}a_\ell.
\end{equation}
In these formulas we used the convention $a_{p+j}= a_{j}$ and  $b_{p+j}= b_{j}$.
\begin{lemma}{\cite[Lem. 6.1, Th. 6.2]{LP}} \label{lem:eta} Set $
({\bf b'}, {\bf a'}) = \eta({\bf a}, {\bf b})
	$. Then
$$
	M(z;{\bf a})	M(z;{\bf b})=M(z;{\bf b'})	M(z;{\bf a'}),
$$
	and $\det 	M(z;{\bf a})= \det 	M(z;{\bf a'})$ and $\det 	M(z;{\bf b})= \det 	M(z;{\bf b'})$.  	Similarly,
$$
N(z;{\bf a})	N(z;{\bf b})=N(z;{\bf b'})	N(z;{\bf a'}),
$$
and $\det 	N(z;{\bf a})= \det 	N(z;{\bf a'})$ and $\det 	N(z;{\bf b})= \det 	N(z;{\bf b'})$.  
	\end{lemma}
To get the switching rule for $M$  and $N$ we 
also need the map
\begin{equation}
\theta({\bf a}, {\bf b}) = ({\bf b'},{\bf a'}),
\end{equation}
where now
\begin{equation}
b'_i = b_{i+1}\frac{a_i+b_i}{a_{i+1}+b_{i+1}} \text{ and } a'_i = a_{i+1}\frac{a_i+b_i}{a_{i+1}+b_{i+1}}.
\end{equation}
\begin{lemma}{\cite[Lem. 6.5, Th 6.6]{LP}} \label{lem:theta} Set $
	({\bf b'}, {\bf a'}) = \theta({\bf a},{\bf b})
	$. Then
		$$
	M(z;{\bf a}) N(z;{\bf b})
	=	N(z;{\bf b'})M(z;{\bf a'}),
$$
	and $\det M(z;{\bf a})= \det M(z;{\bf a'})$ and $\det N(z;{\bf b}) = \det N(z;{\bf b'})$.  
	\end{lemma}
Note that in proving these results, the amount of work is reduced significantly after realizing the following result. 
\begin{lemma}\label{lem:MinversN} For $i,j=1, \ldots, p$ we have
	$$\left(N(a_0,\ldots,a_{p-1})\right)_{ij}=(-1)^{i-j} \left(M(a_0,\ldots,a_{p-1})^{-1}\right)_{ij}.$$
\end{lemma}
Lemmas \ref{lem:eta} and \ref{lem:theta} are sufficient for providing switching rules in case we switch two up jumps, or two down jumps. For switching mixed terms, we need the shift operator 
	\begin{equation}\label{eq:shift}
S(z) = 
\begin{pmatrix}
0 & 0 & \cdots & 0 & z^{-1} \\
1 & 0 & \cdots & 0 & 0 \\
\vdots & \vdots & \ddots & \cdots & \\
0 & 0 & \cdots & 1 & 0
\end{pmatrix},
\end{equation}
and the permutation operator
$$\sigma \left((a_0,\ldots,a_{p-1})\right)= (a_1,\ldots,a_{p-1},a_0).$$
Then the following are straightforward.

\begin{lemma} \label{lem:commutatingshift} We have 
\begin{enumerate}
	\item $S(z)^{-1} M(z;{\bf a}) S(z)= M(z;\sigma({\bf a})). $
	\item $S(z)^{-1} N(z;{\bf a}) S(z) = N(z;\sigma({\bf a})). $
		\item $S(z)^{-1} B({\bf b}) S(z)= B(\sigma({\bf b})). $
\end{enumerate}	
\end{lemma}
\begin{proof}
 Follow easily by direct verification.
\end{proof}
\begin{lemma} \label{lem:mixetatheta} With ${\bf a}^{-1}= \left(a_0^{-1},\ldots,a_{p-1}^{-1}\right) $  we have
	\begin{enumerate}
		\item$
	M(1/z;{\bf a})^T 
		=M(z;{\bf a}^{-1}) S(z)B({\bf a}) 
		= B(\sigma({\bf a})) S(z)M(z;{\bf a}^{-1}). 
			$
		\item $
	N(1/z;{\bf a})^T 
	=N(z;{\bf a}^{-1}) S(z)^{-1}B(\sigma({\bf a}^{-1})) 
	= B({\bf a}^{-1}) S(z)^{-1}N(z;{\bf a}^{-1}). 
	$
	
	\end{enumerate}	
\end{lemma}

\begin{proof}
	Follows by direct verification. Note also that 2. follows from 1. and Lemma \ref{lem:MinversN},
\end{proof}
We are now ready to state an important claim that will be the key to find a factorization of the type \eqref{eq:WHfact} in the setting of our paper. 
\begin{proposition}\label{lem:section_matrix_case:interchange}
	Let $ \phi(z) $ and $ \varphi(z) $ be transition matrices corresponding to a Bernoulli step up or down or a geometric step up or down. Then there are matrices $ \phi'(z) $ and $ \varphi'(z) $ of the same type (meaning that if $\phi(z)$ and $\varphi(z)$ are of the type (2.x) and (2.y) respectively, then so are $\phi'(z)$ and $\varphi'(z)$)  such that
	\begin{equation}
	\phi(z)\varphi(z) = \varphi'(z)\phi'(z),
	\end{equation}
	and
	\begin{equation}
	\det \phi(z) = \det \phi'(z), \text{ and } \det \varphi(z) = \det \varphi'(z).
	\end{equation}
\end{proposition}
\begin{proof}
For switching $\phi^{\cdot,\uparrow}$ and $\phi^{*,\uparrow}$ we use Lemma's \ref{lem:eta} and \ref{lem:theta}. The same lemmas are used for  $\phi^{\cdot,\downarrow}$ and $\phi^{*,\downarrow}$, but now with the transpose everywhere and $z$ replaced by $1/z$. Finally, to switch $\phi^{\cdot,\uparrow}$ and $\phi^{*,\downarrow}$ (and vice versa) we use first Lemma \ref{lem:mixetatheta} and then Lemma's \ref{lem:eta} and \ref{lem:theta}. We leave the details to the reader. 
\end{proof}

	\subsection{Existence of the factorization}
With the switching rules from the previous section  it is now easy to give a constructive proof, for general $p>1$, of the fact that a weight  admits a factorization \eqref{eq:WHfact} precisely when the winding number of the determinant of the weight equals $pM$. 
\begin{theorem}\label{thm:p-periodic_admissible}
	Consider
$
	\phi(z) = \prod_{m=1}^{N}\phi_m(z),
	$
	where  each $ \phi_m $ is as in one of the forms given in \eqref{eq:matrixjump1}-\eqref{eq:matrixjump4} for $ m = 1,\cdots,N $. Then $ \phi $ admits a factorization \eqref{eq:WHfact} with the desired properties if and only if the winding number of $ \det \phi(z) $ with respect to the unit circle equals $pM$.
\end{theorem}
\begin{proof}
	One direction is straightforward. Suppose $ \phi $ has a factorization \eqref{eq:WHfact} with the desired properties. Then using the factorization and the asymptotic behavior at infinity  we obtain  from the argument principle that the winding number  of $ \det \phi $  equals $pM$.
	
	Now let us consider the other direction. We start with  $\phi(z)= \prod_{m=1}^N\phi_m$ and we search for a factorization $\phi= \phi_+ \phi_-$ (we will leave the other factorization to the reader, as it follows by similar arguments).   Each of the $\phi_m$ is one of the four options $\phi^{b,\uparrow}$, $\phi^{b,\downarrow}$, $\phi^{g,\uparrow}$ and $\phi^{g,\downarrow}$ as given in \eqref{eq:matrixjump1}--\eqref{eq:matrixjump4}. Our proof will be relying on a reordering of all these terms based on the switching rules from the previous section. First we will need to investigate which terms are regular inside the circle and which ones outside. 
	
	Let us start with $\phi^{g,\uparrow}(z;{\bf a},{\bf b})$. This term has a pole  at $z=1/a$, but no others. Since we assumed $a<1$ this pole lies outside the circle. Its inverse has a pole at $\infty$ and no others. In other words, $\phi^{g,\uparrow}(z;{\bf a},{\bf b})$ and its inverse are analytic inside the unit  disk. Similarly, the term $\phi^{g,\downarrow}(z;{\bf a},{\bf b})$ has  a pole  at $z=a$, but no others.  Its inverse has a pole at $0$ and no others. This means that $\phi^{g,\downarrow}(z;{\bf a},{\bf b})$ and its inverse are analytic outside the unit  disk. So far we see that we want to switch all $ \phi^{g,\uparrow}(z;{\bf a},{\bf b})$  to the left and all  $\phi^{g,\downarrow}(z;{\bf a},{\bf b})$ to the right. 
	
	The terms  $\phi^{b,\uparrow}(z;{\bf a},{\bf b})$ and $\phi^{b,\downarrow}(z;{\bf a},{\bf b})$ are slightly more subtle. The term $\phi^{b,\uparrow}(z;{\bf a},{\bf b})$ has a pole at $\infty$ and its inverse at   $z=1/a$. Since here  we do not assume $a=a_0\cdots a_{p-1}<1$, the pole $z=1/a$ could be inside and outside. Similarly,  $\phi^{b,\downarrow}(z;{\bf a},{\bf b})$ has a pole at $0$ and its inverse at   $z=a$. 
	
	For now we will ignore the possible singularities at $0$ and $\infty$ and switch all the terms based on the other singularities. Thus, let $ I_1 $ be the $ m $ such that $ \phi_m $ is singular or has a singularity outside the unit circle and $ I_2 $ the $ m $ such that $ \phi_m $ is singular or has a singularity inside the unit circle, with possible exceptions at $z=0$ and $z= \infty$. Note that $ I_1 $ and $ I_2 $ are disjoint and $ I_1\cup I_2 = \{1,\cdots, N\} $. By Lemma \ref{lem:section_matrix_case:interchange} there are $ \varphi_m(z) $, $ m = 1,\cdots,N $, such that
	\begin{equation}\label{eq:phiinvarphi}
	\phi(z) =\prod_{m\in I_1}\varphi_m(z)\prod_{m\in I_2}\varphi_m(z),
	\end{equation}
	and
	\begin{equation}
	\det \phi_m(z) = \det \varphi_m(z).
	\end{equation}
	Note that we can ensure that each type is preserved (i.e. if $\phi_m$ is of the type $\phi^{b,\uparrow}$ then so is $\varphi_m$ but with different parameters). 
	
	This is not yet the factorization that we are after, since the factors (or their inverses) could still contain singularities at $z=0$ and $z= \infty$.  At this point it is important to note that we can turn a Bernoulli step down into a Bernoulli step up (and vice versa), at the cost of a shift as was stated in Lemma \ref{lem:mixetatheta}.	For $m \in I_1$ we write $ \varphi_m = \varphi'_m $  if $ \varphi_m $ is a Bernoulli step down and $ \varphi_m = \varphi'_m S $ otherwise.  Similarly,	for $m \in I_2$ we write  $ \varphi_m = \varphi'_m S ^{-1}$ if $ \varphi_m $ is a Bernoulli step up and $ \varphi_m = \varphi'_m $ otherwise.  A conjugation of $ \varphi_m' $ with $ S $ does not change the determinant or the type, but just shuffles parameters around (see Lemma's \ref{lem:commutatingshift} and \ref{lem:mixetatheta}). We can therefore move all factors $ S $ to one place. That is, there exists $\phi_m'$  (that are of the same type as $\varphi_m'$ and that can be constructed explicitly using Lemma's \ref{lem:commutatingshift} and \ref{lem:mixetatheta} iteratively) such that 
	\begin{equation}\label{eq:section_matrix_case:factorization_of_weight}
	\phi(z)
	=\left(\prod_{m\in I_1}\phi_m'(z)\right)S(z)^{\ell_1-\ell_2} \left(\prod_{m\in I_2}\phi_m'(z)\right),
	\end{equation}
	where 
$$
	\ell_1 = \#\{m\in I_1:\phi_m \text{ corresponds to a Bernoulli step down}\}
$$
and
$$
	\ell_2 = \#\{m\in I_2:\phi_m \text{ corresponds to a Bernoulli step up}\}.
$$
Note $ \phi'_m $ is a Bernoulli or geometric step up when $ m \in I_1 $ and a Bernoulli or geometric step down when $ m \in I_2 $. 
	
  Since  the winding number of 
	\begin{equation}
	\det \phi(z) = z^{(\ell_1-\ell_2)}\det\prod_{m\in I_1}\phi_m'(z)\det \prod_{m\in I_2}\phi_m'(z), 
	\end{equation}
	is equal to $pM$ by assumption, the argument principle implies  that $ \ell_1-\ell_2 = pM $. Now note that $ S(z)^{pM} = z^MI $  and thus 
		\begin{equation}\label{eq:section_matrix_case:factorization_of_weight1}
		\phi(z)
	= z^{M}\prod_{m\in I_1}\phi_m'(z) \prod_{m\in I_2}\phi_m'(z),
	\end{equation} 
	We now find the desired factorization $\phi=\phi_+ \phi_-$ by 
	\begin{equation}
	\phi_+(z) =\left(\prod_{m\in I_1}\phi_m'(z)\right) C,
	\end{equation}
	and
	\begin{equation}
	\phi_-(z) = z^{M} C^{-1} \prod_{m\in I_2}\phi_m'(z),
	\end{equation}
	where $ C $ is a normalizing factor, so that $ \phi_-(z) \to z^{M}I_p$ as $ z \to \infty $. \end{proof}
	
\subsection{Constructing the factorization}\label{sec:method_factorization} 

We finalize our discussion on the factorization \eqref{eq:WHfact} by putting everything together and presenting a general strategy based on the discussion above.

Suppose that we want to construct  a factorization $\phi= \phi_+\phi_-$ as in \eqref{eq:WHfact} for the  product

\begin{equation} \label{eq:prodphiminABa}
	\prod_{m=1}^N \phi_m= A_1^{(1)}(z)B_1^{(1)} (z) A_2^{(1)}(z) B_2^{(1)}(z) \cdots A_k^{(1)}(z) B_k^{(1)}(z),
\end{equation}
where $A_j^{(1)}$ are regular inside the disk, with possibly a singularity at zero (for the matrix and the inverse) and the $B_j^{(0)}$ outside, with possibly a singularity at infinity. Each $A_j^{(1)}$ and $B_j^{(1)}$ is the product of $\phi_m$'s of the form \eqref{eq:matrixjump1}--\eqref{eq:matrixjump4}. We can iteratively switch the matrices  in $k$ steps, where in the $\ell$-th step we simultaneously switch $k-\ell$ pairs according to the rule
	$$  B^{(\ell)}_{j}(z) A^{(\ell)}_{j+\ell}(z)= A^{(\ell+1)}_{j+\ell}(z)B^{(\ell+1)}_{j}(z),\quad  j=\ell,\ldots, k-\ell-1,,$$
	with  the matrices $A^{(\ell+1)}_{j+\ell}(z)$, $B^{(\ell+1)}_{j}(z)$ chosen according to the rules from Section \ref{sec:switching}.	After the last step we obtain 
	\begin{equation} \label{eq:prodphiminABb}
	\prod_{m=1}^N \phi_m=A_1^{(1)}A_2^{(2)}A_3^{(3)}\cdots A_k^{(k)} B_1^{(k)}\cdots B_{k-2}^{(3)}B_{k-1}^{(2)}B_k^{(1)},
\end{equation}
	which is almost the desired factorization  $\phi= \phi_+ \phi_-$. We still need to take care of the possible singularities at $z=0$ and have the right behavior at $z=\infty$. We are now at the same point as \eqref{eq:phiinvarphi} in the proof of Theorem \ref{thm:p-periodic_admissible}. For the final steps, we need the number $\ell_1$ of Bernoulli steps down among the $A_j^{(k)}$'s and the number $\ell_2$ of Bernoulli steps up among the $B_j^{(k)}$. Then 
		\begin{equation} \label{eq:prodphiminABbfinal}
\phi_+=A_1^{(1)}A_2^{(2)}A_3^{(3)}\cdots A_k^{(k)} S^{-\ell_1} C\quad 
\phi_-= C^{-1} S^{\ell_2} B_1^{(k)}\cdots B_{k-2}^{(3)}B_{k-1}^{(2)}B_k^{(1)},
	\end{equation}
	where $C$ is constant matrix to ensure that the leading coefficient of the expansion of  $\phi_-$ as $z\to \infty $ equals $I_p$.
	
Note that the procedure simplifies in case we have \eqref{eq:phidoubleper}. In that case, we can take
$$A_j^{(1)}= \Phi_+, \text{ and } B_j^{(1)}= \Phi_-,$$
where $\Phi=\Phi_+ \Phi_-$ and $\Phi_+$ and $\Phi_-$ are regular inside and outside the disk respectively.   Indeed, in this case the $A_j^{(\ell)}$  are the same for $j=\ell, \ldots, k$. Similarly,  $B_j^{(\ell)}$  are the same for $j=1, \ldots, k-\ell+1$. Thus less administration has to be taken care of in this case. We will write $A_j^{(\ell)}= \Phi_+^{(\ell)}$ and   $B^{(\ell)}_j= \Phi_-^{(\ell)}$  to indicate that there is no dependence on $j$. Then \eqref{eq:prodphiminABb} becomes 
$$\prod_{m=1}^N \phi_m=\Phi_+^{(1)}\Phi_+^{(2)}\Phi_+^{(3)}\cdots \Phi_+^{(k)} \Phi_-^{(k)}\cdots \Phi_-^{(3)}\Phi_-^{(2)}\Phi_-^{(1)}.$$
The examples that we will discuss in the rest of this paper are all of this type, but they even have an important additional simplification. In all of the examples it will be true that after a few iterations we return to the initial situation. That is, 
\begin{equation}\label{eq:finite_steps}
\Phi_+^{(\ell+q)}= \Phi_+^{(\ell)},
\end{equation}
for some $q$ and all $\ell$. This means  that \eqref{eq:prodphiminABb}  can be written as 
\begin{equation}\label{eq:closed_formula}
\prod_{m=1}^N \phi_m=\left(\Phi_+^{(1)}\Phi_+^{(2)}\cdots\Phi_+^{(q)}\right)^{N/q}\left( \Phi_-^{(q)}\cdots\Phi_-^{(2)}\Phi_-^{(1)}\right)^{N/q},
\end{equation}
where we, for simplicity, have assume that  $N\equiv 0 \mod q$.  By including the factors $C$ and $S^{\ell_j}$ as in \eqref{eq:prodphiminABbfinal} we find the desired factorization $\phi=\phi_+\phi_-$.

This settles the factorization of $\phi=\phi_+ \phi_-$, and, naturally, the other factorization $\phi= \tilde \phi_- \tilde \phi_+$ follows a similar discussion.

\section{Example: Domino tiling of the Aztec diamond}\label{section:example_aztec_diamond}

Domino tilings of the Aztec diamond is a well-studied topic introduced in \cite{EKLP}. In \cite{CJY,CEP,JPS95,J02,J05} local and global properties in the large $ N $ limit are analyzed, both for the uniform weight and a weighting that favors either the vertical or horizontal dominoes over the other (see the setup in Section \ref{sec:aztec_diamond:p=1}). Also asymptotic results for models with periodic weighting have been studied. In \cite{KO,KOS} global properties have been discussed in a general context, including  periodic weightings for the Aztec diamond. See also \cite{FR} for results for a certain family of periodic weightings of the Aztec diamond.  To the best of our knowledge, only in case of the two-periodic weighting the fine asymptotic properties were studied. The first results were based on computation of the inverse Kasteleyn matrix, \cite{BCJ,CJ,CY}.   Recently , \cite{DK} used a connection to matrix-valued orthogonal polynomials. Here we will give an alternative derivation for the double integral formula in \cite{DK} using the machinery of the present paper. We will also include an example with higher periodicity. 


We start by recalling the connection between the Aztec diamond and non-intersecting paths, which has been discussed and used  many times before.  We will therefore be brief in our explanation, and refer to \cite{J02}, and also \cite{DK,J05}. Note, however, there are minor differences in the construction of the paths, which will be of help when  we take the number of paths to infinity. 

The $ N\times N $ Aztec diamond is a certain region consisting of $ 2N(N+1) $ squares, which we usually color in a white/black chess board way, see Figure \ref{fig:section_aztec_diamond:boundary}. A domino tiling of the Aztec diamond is a configuration of tiles, $ 2\times 1 $ or $ 1 \times 2  $ rectangles, which covers the Aztec diamond such that no dominoes  overlap, see Figure \ref{fig:section_aztec_diamond:boundary}. The tiles are divided into four types, North, West, South and East. To each tile in a tiling we associate a weight which depends on the type of tile and the position of the tile. To each tiling $ \mathcal T $ of the Aztec diamond we then associate a weight $ w(\mathcal T) $ which is the product of the weights associated to the tiles in $ \mathcal T $. A natural probability measure on the space of all possible tilings $ \mathcal T $ of the Aztec diamond is given by
\begin{equation}\label{eq:section_aztec_diamond:measure}
\mathbb P[\mathcal T] = \frac{w(\mathcal T)}{\sum_{\tilde {\mathcal T}}w(\tilde{\mathcal T})},
\end{equation}
where the sum is over all possible tilings.

\begin{figure}[t]
	\begin{center}
		\begin{tikzpicture}[scale=0.5]
		
		\foreach \x/\y in {-4/-1,-3/-2,-2/-3,-1/-4,
			-3/0,-2/-1,-1/-2,0/-3,
			-2/1,-1/0,0/-1,1/-2,
			-1/2,0/1,1/0,2/-1,
			0/3,1/2,2/1,3/0}
		{\fill[outer color=lightgray,inner color=gray]
			(\x,\y) rectangle (\x+1,\y+1); 
		}

		\foreach \x/\y in {-4/-1,-3/-2,-2/-3,-1/-4}
		{\draw [line width = 1mm]
			(\x+1,\y)--(\x,\y)--(\x,\y+1); 
		}
		
		\foreach \x/\y in {-4/1,-3/2,-2/3,-1/4}
		{\draw [line width = 1mm]
			(\x,\y-1)--(\x,\y)--(\x+1,\y); 
		}
		
		\foreach \x/\y in {4/1,3/2,2/3,1/4}
		{\draw [line width = 1mm]
			(\x-1,\y)--(\x,\y)--(\x,\y-1); 
		}
		
		\foreach \x/\y in {4/-1,3/-2,2/-3,1/-4}
		{\draw [line width = 1mm]
			(\x-1,\y)--(\x,\y)--(\x,\y+1); 
		}
		
		\draw[outer color=lightgray,inner color=gray]
		(5,0.5)--(5,-0.5)--(6,-0.5)--(6,0.5);
		\draw [line width = 1mm] (5,-0.5) rectangle (6,1.5);
		\draw (5.5,-1) node {West};
		
		\draw[outer color=lightgray,inner color=gray]
		(6.5,3)--(7.5,3)--(7.5,4)--(6.5,4);
		\draw[line width = 1mm] (5.5,3) rectangle (7.5,4);
		\draw (6.5,2.5) node {North};
		
		\draw[outer color=lightgray,inner color=gray]
		(6.5,-3)--(5.5,-3)--(5.5,-2)--(6.5,-2);
		\draw[line width = 1mm] (5.5,-3) rectangle (7.5,-2);
		\draw (6.5,-3.5) node {South};
		
		\draw[outer color=lightgray,inner color=gray]
		(7.5,0.5)--(7.5,1.5)--(8.5,1.5)--(8.5,0.5);
		\draw[line width = 1mm] (7.5,-0.5) rectangle (8.5,1.5);
		\draw (8,-1) node {East};
		\end{tikzpicture}
		\begin{tikzpicture}[scale=0.5]
		
		\foreach \x/\y in {-4/0,-3/1,-2/2,-3/-1,-1/-1,0/-2}
		{ \fill[outer color=lightgray,inner color=gray]
			(\x,\y) rectangle (\x+1,\y+1); 
			\draw [line width = 1mm] (\x,\y) rectangle (\x+1,\y+2);
		}
		
		\foreach \x/\y in {-2/-1,1/-2,2/-1,3/0,2/1,1/2}
		{  \fill[outer color=lightgray,inner color=gray]
			(\x,\y+1) rectangle (\x+1,\y+2);
			\draw [line width = 1mm] (\x,\y) rectangle (\x+1,\y+2);
		}
		
		\foreach \x/\y in {-1/-3,-2/-2,-1/3,0/0}
		{  \fill[outer color=lightgray,inner color=gray]
			(\x,\y) rectangle (\x+1,\y+1);
			\draw [line width = 1mm] (\x,\y) rectangle (\x+2,\y+1);
		}
		
		\foreach \x/\y in {-2/1,0/1,-1/2,-1/4}
		{  \fill[outer color=lightgray,inner color=gray]
			(\x+1,\y) rectangle (\x+2,\y+1);
			\draw [line width=1mm] (\x,\y) rectangle (\x+2,\y+1);
		}
		\end{tikzpicture}
		\caption{The boundary of the $ 4 \times 4 $ Aztec diamond together with the four different tiles and an example of a tiling of the Aztec diamond.
			\label{fig:section_aztec_diamond:boundary} } 
	\end{center}
\end{figure}
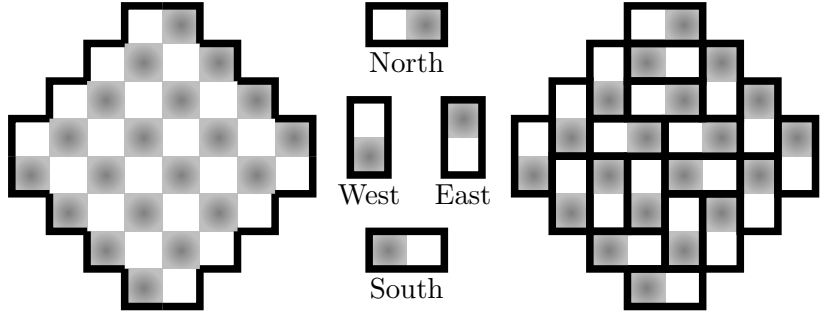

To specify the weights on the tiles we introduce a coordinate system such that the colored squares cover exactly the points $ (m-x+N-1,m+x) $ for $ m=0,1,\cdots,N $ and $ x=0,1,\cdots,N-1 $. For a tile with the black part on $ (m-x+N-1,m+x) $ we set the weight to $ b_{m,x}\in(0,\infty) $ if it is a West tile, to $ c_{m,x}\in (0,\infty) $ if it is a South tile, to $ d_{m,x}\in(0,\infty) $ if it is an East tile and, without loss of generality, to one if it is a North tile. To connect this to non-intersecting paths which fits into the framework of Section \ref{sec:finite}, we assume periodicity in $x$, that is there is an integer $ p $ such that $ b_{m,x+p}=b_{m,x} $, $ c_{m,x+p}=c_{m,x} $ and $ d_{m,x+p}=d_{m,x} $ for all $ m $ and $ x $.

\begin{figure}
	\begin{center}
		\begin{tikzpicture}[scale=0.5]
		\tikzset{->-/.style={decoration={
					markings,
					mark=at position .5 with {\arrow{stealth}}},postaction={decorate}}}
		\foreach \y in {-4,-3,-2,-1,0,1,2,3,4}
		{
			\foreach \x in {0,1,2,3,4,5,6,7,8,9}
			{\draw[->,>=stealth] (\x-.5,\y)--(\x+.5,\y);
				\draw (9.5,\y)--(10.5,\y);
		}}
		\foreach \x in {0,2,4,6,8,10}
		{\foreach \y in {-4,-3,-2,-1,0,1,2,3}
			{\draw[-<,>=stealth] (\x,\y-.5)--(\x,\y+.5);
				\draw (\x,3.5)--(\x,4.5);}}
		\foreach \y in {-3,-2,-1,0,1,2,3,4}
		{\foreach \x in {0,2,4,6,8}
			\draw [->-](\x+1,\y)to(\x+2,\y-1);}
		\draw (1,-5) node {$m=0$};
		\draw (9,-5) node {$m=2N$};
		\draw (-1,0) node {0};
		\draw (-1.5,3) node {$n-1$};
		\foreach \y in {0,1,2,3}
		{\draw (1,\y) node[circle,fill,inner sep=2pt]{};
			\draw (9,\y-4) node[circle,fill,inner sep=2pt]{};}
		\end{tikzpicture}
		\hspace{.5cm}
		\begin{tikzpicture}[scale=0.3]
		\foreach \x in {-1,1,3,5,7,9}
		{\draw (\x,-5.5)--(\x,11.5);}
		\foreach \y in {-5,-4,...,11}
		{\draw (-1.5,\y)--(9.5,\y);}
		
		{\draw (-1.5,0) node[left] {0};
		\draw (-1.5,10) node[left] {$n-1$};
		\draw (0,-5.5) node[below] {$m=0$};
		\draw (8,-5.5) node[below] {$m=2N$};}
		\foreach \x in {0,2,4,6,8}
		{\foreach \y in {-5,-4,...,10}
			{\draw (\x,\y+1)--(\x+1,\y);}}
		
		\foreach \y in {0,1,...,10}
		{\draw (0,\y) node[circle,fill,inner sep=2pt]{};
			\draw (8,\y-4) node[circle,fill,inner sep=2pt]{};}
		
		\foreach \x/\y in {0/7,1/6,0/8,1/8,3/6,0/9,1/9,3/8,4/8,5/6,0/10,1/10,2/10,3/10,4/10,5/10,6/10,7/6}
		{\draw[line width = .7mm] (\x,\y)--(\x+1,\y);
			\draw (\x+1,\y) node[circle,fill,inner sep=1.6pt]{};}
		\foreach \x/\y in {1/7,3/7,5/8,5/7,7/10,7/9,7/8,7/7}
		{\draw[line width = .7mm] (\x,\y)--(\x,\y-1);}
		\foreach \x/\y in {2/8,2/9}
		{\draw[line width=.7mm] (\x,\y)--(\x+1,\y-1);
			\draw (\x+1,\y-1) node[circle,fill,inner sep=1.6pt]{};}
		\foreach \x in {0,2,4,6}
		{\foreach \y in {1,2,...,6}
			{\draw[line width = .7mm] (\x,\y)--(\x+1,\y-1)--(\x+2,\y-1);
				\draw (\x+1,\y-1) node[circle,fill,inner sep=1.6pt]{};
				\draw (\x+2,\y-1) node[circle,fill,inner sep=1.6pt]{};}}
		\foreach \x/\y in {7/-1,7/-2,6/-2,5/-2,7/-3,6/-3,5/-3,3/-2,7/-4,6/-4,5/-4,3/-3,1/-2}
		{\draw[line width = .7mm] (\x,\y)--(\x+1,\y);
			\draw (\x+1,\y) node[circle,fill,inner sep=1.6pt]{};}
		\foreach \x/\y in {5/-1,3/-1,1/-1}
		{\draw[line width = .7mm] (\x,\y)--(\x,\y-1);}
		\foreach \x/\y in {0/0,2/0,4/0,6/0,4/-2,4/-3,2/-2}
		{\draw[line width=.7mm] (\x,\y)--(\x+1,\y-1);
			\draw (\x+1,\y-1) node[circle,fill,inner sep=1.6pt]{};}
		\end{tikzpicture}
		\caption{The underlying graph and a configuration of non-intersecting paths with $ n=11 $ and $ 2N=8 $. The top part of the right picture corresponds to a $ 4\times 4 $ Aztec diamond.
			\label{fig:section_aztec_diamond:graph} }
	\end{center}
\end{figure}
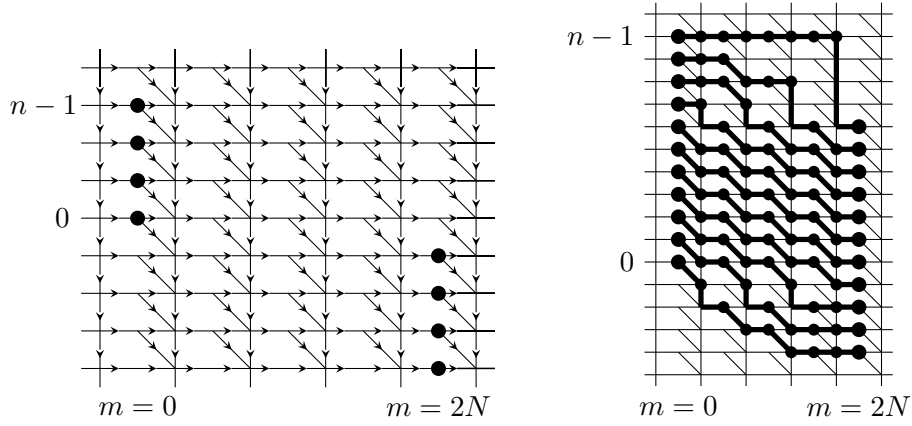

Now, let us leave the subject of domino tilings of the Aztec diamond for a moment and consider instead a directed graph, as in Figure \ref{fig:section_aztec_diamond:graph}. We distribute weights on each of the edges. Then  consider the set $ \Omega $ of all paths starting at $ (0,0), (0,1),\cdots, (0,n-1) $ and ending at $ (2N,-N),(2N,-N+1),\cdots,(2N,-N+n-1) $ for some $ n \geq N $. We assign a weight to each collection of paths $\omega \in \Omega$  by taking the product of the weights of the edges that lie on one of the paths. This defines a probability measure on $\Omega$ by taking the probability of having $\omega$ proportional to its weight. We are now interested in the restriction of this probability measure to the set $\Omega_{n.i.}$ consisting of collections of paths that do not intersect.   The Lindstr\"om-Gessel-Viennot Theorem \cite{GV,L} gives us, as discussed in Section \ref{sec:finite}, a probability measure on the space $\Omega_{n.i}$ of the type \eqref{eq:productofdeterminants}. If we choose the edges weights in a specific way, this probability measure  can be related to the measure \eqref{eq:section_aztec_diamond:measure} on the Aztec diamond.

\begin{figure}
	\begin{center}
		\begin{tikzpicture}[rotate=-45, scale=0.6]
		
		\foreach \x/\y in {-4/0,-3/1,-2/2,-3/-1,-1/-1,0/-2}
		{ \fill[outer color=lightgray,inner color=gray]
			(\x,\y) rectangle (\x+1,\y+1); 
			\draw [line width = 1mm] (\x,\y) rectangle (\x+1,\y+2);
			\draw[very thick,red] (\x,\y+.5)--(\x+1,\y+1.5);
		}
		
		\foreach \x/\y in {-2/-1,1/-2,2/-1,3/0,2/1,1/2}
		{  \fill[outer color=lightgray,inner color=gray]
			(\x,\y+1) rectangle (\x+1,\y+2);
			\draw [line width = 1mm] (\x,\y) rectangle (\x+1,\y+2);
			\draw[very thick,red] (\x,\y+1.5)--(\x+1,\y+.5);
		}
		
		\foreach \x/\y in {-1/-3,-2/-2,-1/3,0/0}
		{  \fill[outer color=lightgray,inner color=gray]
			(\x,\y) rectangle (\x+1,\y+1);
			\draw [line width = 1mm] (\x,\y) rectangle (\x+2,\y+1);
			\draw[very thick,red] (\x,\y+.5)--(\x+2,\y+.5);
		}
		
		\foreach \x/\y in {-2/1,0/1,-1/2,-1/4}
		{  \fill[outer color=lightgray,inner color=gray]
			(\x+1,\y) rectangle (\x+2,\y+1);
			\draw [line width=1mm] (\x,\y) rectangle (\x+2,\y+1);
		}
		\end{tikzpicture}
		\hspace{.5cm}
		\begin{tikzpicture}[scale=0.45]
		\foreach \x in {-1,0,...,9}
		{\draw (\x,-3.5) node[below] {$\x$};}
		\foreach \x in {-1,1,3,5,7,9}
		{\draw (\x,-3.5)--(\x,5.5);}
		\foreach \y in {-4,-3,-2,-1,0,1,2,3,4}
		{\draw (-1.5,\y+1)--(9.5,\y+1);}
		\foreach \x in {0,2,4,6,8}
		{\foreach \y in {-3,-2,-1,0,1,2,3,4}
			{\draw (\x,\y+1)--(\x+1,\y);}}
		
		\foreach \x/\y in {0/4,2/4,4/4,0/3,2/2,2/1}
		{
			\draw[line width = 1mm,red] (\x,\y)--(\x+1,\y);
		}
		
		\foreach \x/\y in {7/3,7/2,7/1,1/3,5/1,3/1}
		{
			\draw[line width = 1mm,red] (\x,\y)--(\x,\y-1);
		}
		
		\foreach \x/\y in {6/4,4/2,0/2,0/1}
		{
			\draw[line width = 1mm,red] (\x,\y)--(\x+1,\y-1);
		}
		
		\foreach \x/\y in {1/4,3/4,5/4,1/2,3/2,1/1}
		{
			\draw[line width = 1mm,red] (\x,\y)--(\x+1,\y);
		}
		
		\foreach \y in {-3,-2,-1,0,1,2,3,4}
		{\draw (0,\y) node[circle,fill,inner sep=2pt]{};}
		\foreach \y in {-3,-2,-1,0}
		{\draw (8,\y) node[circle,fill,inner sep=2pt]{};}
		\foreach \x in {1,3,5,7}
		{\draw (\x,0) node[circle,fill=red,inner sep=2pt]{};}
		\foreach \y in {-2,-1,0}
		{\foreach \x in {1,3,5,7}
			{\draw[line width = .7mm] (\x-1,\y)--(\x,\y-1);
				\draw[line width = .7mm] (\x,\y)--(\x+1,\y);
				\draw[line width = .7mm] (\x,-3)--(\x+1,-3);}}
		\end{tikzpicture}
		\caption{The left picture is the tiling of the Aztec diamond in Figure \ref{fig:section_aztec_diamond:boundary} rotated clockwise by $ \frac{\pi}{4} $ together with the corresponding paths. The right picture is an element in $ \Omega_{n.i.} $ corresponding to the same tiling of the Aztec diamond. It turns out that the red part is independent from the black part. \label{fig:section_aztec_diamond:tiling_paths}  }
	\end{center}
\end{figure}
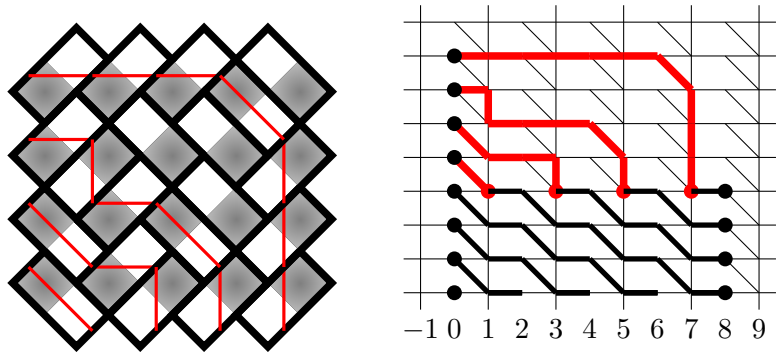

Namely, let $ \mathcal T $ be a tiling of the Aztec diamond. Draw lines on the West, South and East tiles, according to
\begin{center}
	\tikz[scale=.6]{ \fill[outer color=lightgray,inner color=gray](0,0) rectangle (1,1);
		\draw [line width = 1mm] (0,0) rectangle (1,2);
		\draw[very thick,red] (0,.5)--(1,1.5);} \quad
	\tikz[scale=.6]{\fill[outer color=lightgray,inner color=gray](0,0) rectangle (1,1);
		\draw [line width = 1mm] (0,0) rectangle (2,1);
		\draw[very thick,red] (0,.5)--(2,.5);}
	\quad and \quad 
	\tikz[scale=.6]{\fill[outer color=lightgray,inner color=gray](0,1) rectangle (1,2);
		\draw [line width = 1mm] (0,0) rectangle (1,2);
		\draw[very thick,red] (0,1.5)--(1,0.5);},
\end{center}
and rotate the Aztec diamond clockwise by $ \frac{\pi}{4} $ (Figure \ref{fig:section_aztec_diamond:tiling_paths}). We obtain a family of non-intersecting paths. These non-intersecting paths correspond to the top part of an element in $ \Omega_{n.i.} $, the red part in Figure \ref{fig:section_aztec_diamond:tiling_paths}. The only difference, as can be seen in Figure \ref{fig:section_aztec_diamond:tiling_paths}, is the horizontal part of each path going from an odd step to the consecutive even step. We  add these steps artificially (giving them weight one). 

To obtain all paths in an element in $ \Omega_{n.i.} $, we consider tilings on a bigger domain.  Consider a boundary as in Figure \ref{fig:extended_boundary_aztec_diamond}, the boundary of an $ N\times N $ Aztec diamond and an $ (N-1)\times (N-1) $ Aztec diamond connected by a long diagonal ``corridor''. Any tiling of this domain is actually decoupled into three parts. In fact, the ``corridor''-part can only be tiled by Souths dominoes, which implies that a tiling of this bigger domain contains a tiling of an $N\times N $ Aztec diamond (and a tiling of an $(N-1)\times (N-1) $ Aztec diamond). So a probability measure on the space of tilings on this bigger domain can be viewed as a probability measure on the $ N\times N $ Aztec diamond. By adding paths on the tilings, as we did before, rotate the domain clockwise by $ \frac{\pi}{4} $ and add the artificial lines, we obtain an element in $\Omega_{n.i.} $ (Figure \ref{fig:extended_boundary_aztec_diamond}).

If we do not include the ``corridor'' this is similar to the construction done in \cite{DK}. However, it is essential that we do include the ``corridor'', since the number of paths  need to grow to infinity for our results to apply.

\begin{center}
	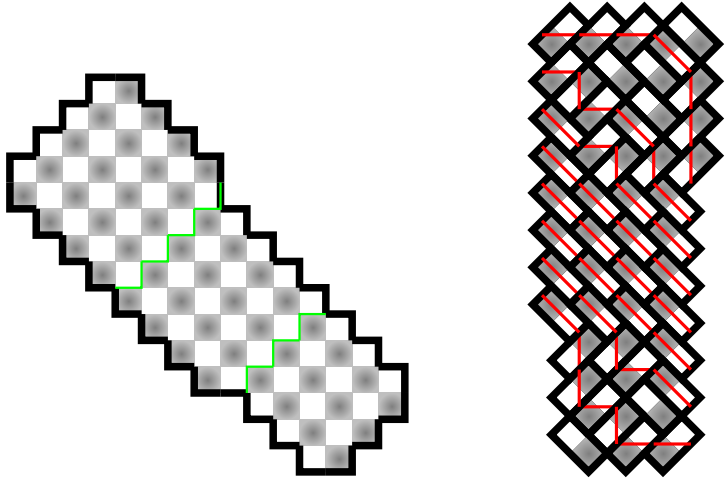
\begin{figure}[t]
		\centering
		\begin{subfigure}[t]{0.45\textwidth}
			\begin{tikzpicture}[scale=.35]
			\foreach \x/\y in {-4/-1,-3/-2,-2/-3,-1/-4,
				-3/0,-2/-1,-1/-2,0/-3,
				-2/1,-1/0,0/-1,1/-2,
				-1/2,0/1,1/0,2/-1,
				0/3,1/2,2/1,3/0,
				0/-5,1/-4,2/-3,3/-2,
				1/-6,2/-5,3/-4,4/-3,
				2/-7,3/-6,4/-5,5/-4,
				3/-8,4/-7,5/-6,6/-5,
				5/-8,6/-7,7/-6,
				6/-9,7/-8,8/-7,
				7/-10,8/-9,9/-8,
				8/-11,9/-10,10/-9}
			{\fill[outer color=lightgray,inner color=gray]
				(\x,\y) rectangle (\x+1,\y+1); 
			}

			\foreach \x/\y in {-4/-1,-3/-2,-2/-3,-1/-4,0/-5,1/-6,2/-7,3/-8,5/-9,6/-10,7/-11}
			{\draw [line width = 1mm]
				(\x+1,\y)--(\x,\y)--(\x,\y+1); 
			}
			
			\foreach \x/\y in {-4/1,-3/2,-2/3,-1/4}
			{\draw [line width = 1mm]
				(\x,\y-1)--(\x,\y)--(\x+1,\y); 
			}
			
			\foreach \x/\y in {4/1,3/2,2/3,1/4,5/-1,6/-2,7/-3,8/-4,9/-5,10/-6,11/-7}
			{\draw [line width = 1mm]
				(\x-1,\y)--(\x,\y)--(\x,\y-1); 
			}
			
			\foreach \x/\y in {9/-11,10/-10,11/-9}
			{\draw [line width = 1mm]
				(\x-1,\y)--(\x,\y)--(\x,\y+1); 
			}
			
			\draw [line width=1mm] (4,-8)--(5,-8);
			\draw [line width=1mm] (4,0)--(4,-1);
			
			\foreach \x/\y in {4/-1,3/-2,2/-3,1/-4}
			{\draw [green,line width = .3mm]
				(\x-1,\y)--(\x,\y)--(\x,\y+1); 
			}
			\foreach \x/\y in {5/-7,6/-6,7/-5}
			{\draw [green,line width = .3mm]
				(\x,\y-1)--(\x,\y)--(\x+1,\y); 
			}
			
			\end{tikzpicture}
		\end{subfigure}
		\hspace{1cm}
		\begin{subfigure}[t]{0.45\textwidth}
			\begin{tikzpicture}[rotate=-45, scale=0.35]
			
			\foreach \x/\y in {-4/0,-3/1,-2/2,-3/-1,-1/-1,0/-2,
				7/-9,9/-9,10/-8,7/-7}
			{ \fill[outer color=lightgray,inner color=gray]
				(\x,\y) rectangle (\x+1,\y+1); 
				\draw [line width = 1mm] (\x,\y) rectangle (\x+1,\y+2);
				\draw[very thick,red] (\x,\y+.5)--(\x+1,\y+1.5);
			}
			
			\foreach \x/\y in {-2/-1,1/-2,2/-1,3/0,2/1,1/2,
				5/-8,6/-7,6/-9,8/-9}
			{  \fill[outer color=lightgray,inner color=gray]
				(\x,\y+1) rectangle (\x+1,\y+2);
				\draw [line width = 1mm] (\x,\y) rectangle (\x+1,\y+2);
				\draw[very thick,red] (\x,\y+1.5)--(\x+1,\y+.5);
			}
			
			\foreach \x/\y in {-1/-3,-2/-2,-1/3,0/0,
				0/-4,1/-5,2/-6,3/-7,1/-3,2/-4,3/-5,4/-6,2/-2,3/-3,4/-4,5/-5,3/-1,4/-2,5/-3,6/-4,
				7/-5,8/-6}
			{  \fill[outer color=lightgray,inner color=gray]
				(\x,\y) rectangle (\x+1,\y+1);
				\draw [line width = 1mm] (\x,\y) rectangle (\x+2,\y+1);
				\draw[very thick,red] (\x,\y+.5)--(\x+2,\y+.5);
			}
			
			\foreach \x/\y in {-2/1,0/1,-1/2,-1/4,
				8/-7,7/-10}
			{  \fill[outer color=lightgray,inner color=gray]
				(\x+1,\y) rectangle (\x+2,\y+1);
				\draw [line width=1mm] (\x,\y) rectangle (\x+2,\y+1);
			}
			\end{tikzpicture}
		\end{subfigure}
		\caption{It turns out that a tiling of the extended domain can never have a tile that crosses the green lines indicated in the picture. To see this it may help to take the view from the non-intersecting paths perspective. \label{fig:extended_boundary_aztec_diamond}}
	\end{figure}
\end{center}

The above discussion leads to the following proposition.
\begin{proposition}\label{prop:section_aztec_diamond:bijection}
	The probability measure on the pace of $ pn $ non-intersecting paths, $\Omega_{n.i.} $, defined by 
	\eqref{eq:productofdeterminants}, \eqref{eq:endpoints} with $ pM=-N$ and with transition matrices $ T_m $, $ m=1,2,\cdots,2N $, where $ T_{2m'+1}(x,x) = b_{m',x-pn+N} $, $ T_{2m'+1}(x,x-1) = c_{m',x-pn+N} $, $ T_{2m'}(x,x) = 1 $, $ T_{2m'}(x,y) = \prod_{k=x}^{y+1}d_{m',k-pn+N} $ if $ y < x $, and zero otherwise, can be viewed as a measure on the $ N\times N $ Aztec diamond. Moreover this measure, when viewed as a measure on the $ N\times N $ Aztec diamond, is the same as the measure \eqref{eq:section_aztec_diamond:measure}.
\end{proposition}

We will be interested in the top part of the paths, since it is the $ N\times N $ Aztec diamond we want to study. That is, in the limit when $ n \to \infty $, taking an infinitely long ``corridor'', the part of the correlation kernel we consider converges to the correlation kernel $ K_{top} $ in Theorem \ref{thm:main}. In fact, with the change of variables \eqref{eq:shiftininparam}, it is the part $ 0< m,m' < 2N $, $ -N/p \leq \xi,\xi' \leq -1 $ which corresponds to the Aztec diamond. For simplicity we will consider Aztec diamonds of size $ pN\times pN $ and then replace $ N $ by $ pN $.

\subsection{Domino tiling of the Aztec diamond $ p=1 $}\label{sec:aztec_diamond:p=1}
The easiest possible choice is to take all weights on the tiles equal to one,  leading to the uniform measure on the Aztec diamond. This does not fit directly into our framework, since the weight has both poles and singularities on the unit circle. Therefore (as was done in \cite{J05}) we introduce a parameter $ a \in (0,1) $ and take the weight on the tiles to $ b_{m,x}=1 $ and $ c_{m,x}^{-1} = d_{m,x} = a $ and later take the limit $ a \to 1 $. That is, consider the probability measure defined by \eqref{eq:productofdeterminants}, \eqref{eq:endpoints} and with transition matrix $ T_m = T_{\phi_m} $, \eqref{eq:Tmblocktoeplitz}, where 
\begin{equation}
\phi_m(z)=
\begin{cases} 1+a^{-1}z^{-1}, &m \text{ odd,}\\
\frac{1}{1-az^{-1}}, & m \text{ even,}
\end{cases}
\end{equation}
and $ M=-N $. This is actually an interesting model by itself, a model which favors vertical dominoes over horizontal dominoes. It is a determinantal point process with correlation kernel given by
\begin{multline}
K(2m,\xi;2m',\xi') = - a^{m'-m}\frac{\chi_{m>m'}}{2\pi \i}\int_{\gamma_{\text{int}}}\left(\frac{az+1}{z-a}\right)^{m-m'}z^{\xi'-\xi}\frac{\d z}{z} \\ 
+ a^{m'-m}\frac{1}{(2\pi\i)^2}\int_{\gamma_{\text{int}}}\int_{\gamma_\text{ext}} \frac{w^{\xi'+N}}{z^{\xi+N+1}} \frac{\left(w-a\right)^{m'-N}}{(a w+1)^{m'}}\frac{(az+1)^m}{\left(z-a\right)^{m-N}}\frac{\d z\d w}{z-w},
\end{multline}
where $ \gamma_\text{int} $ is a circle around zero and one and $ \gamma_\text{ext} $ is a circle around $ \gamma_\text{int} $ with $ -1 $ in its exterior. By analyticity of the kernel we may take $ a\to 1 $. So the kernel for the uniform measure is given by the above formula with $  a=1 $ (see \cite{J05}).

\subsection{Domino tiling of the Aztec diamond $ p=2 $}\label{sec:aztec_diamond:p=2}
The two-periodic Aztec diamond is, to our knowledge, the only non-intersecting path model with block Toeplitz transition matrices which has been studied in full extent. 

The two-periodic Aztec diamond is given in the following way. Let $ \alpha,\beta > 0 $ with $ \alpha\beta = 1 $ and set $ b_{m,x} = 1 $ for all $ m $ and $ x $, $ c_{m,x} = d_{m+1,x} = \alpha^2 $ if $ x $ is odd and $ m $ is even, $ c_{m,x} = d_{m+1,x} = \beta^2 $ if $ x $ is even and $ m $ is even, $ c_{m,x} = d_{m+1,x} = 1 $ otherwise. For simplicity we consider the Aztec diamond of size $ 2N $. That is, consider the probability measure defined by \eqref{eq:productofdeterminants}, \eqref{eq:endpoints} with transition matrix $ T_m = T_{\phi_m} $, \eqref{eq:Tmblocktoeplitz}, where 
\begin{equation}
\phi_{4k+1}(z)  =  
\begin{pmatrix}
1 & \alpha^2 z^{-1}\\
\beta^2 & 1
\end{pmatrix},
\end{equation}
\begin{equation}
\phi_{4k+2}(z) = \frac{1}{1-z^{-1}} 
\begin{pmatrix}
1 & \alpha^2 z^{-1}\\
\beta^2 & 1
\end{pmatrix}, 
\end{equation}
\begin{equation}
\phi_{4k+3}(z) = 
\begin{pmatrix}
1 & z^{-1}\\
1 & 1
\end{pmatrix}, 
\end{equation}
\begin{equation}
\phi_{4k+4}(z) = \frac{1}{1-z^{-1}} 
\begin{pmatrix}
1 & z^{-1}\\
1 & 1
\end{pmatrix}, 
\end{equation}
for $ k = 0,1,\cdots,N-1 $ and with $ M = -N $. Then
\begin{equation}\label{eq:section_aztec_diamond:2weight}
\phi(z) = \frac{1}{(1-z^{-1})^{2N}}\left(
\begin{pmatrix}
1 & \alpha^2 z^{-1}\\
\beta^2 & 1
\end{pmatrix}^2
\begin{pmatrix}
1 & z^{-1}\\
1 & 1
\end{pmatrix}^2
\right)^N.
\end{equation}

\begin{theorem}\label{thm:two_periodic_aztec_diamond}
	Consider the two-periodic Aztec diamond defined above with $ N $ even. This model is a determinantal point process with correlation kernel given by 
	\begin{multline}
	\left[K(4m,2\xi+i;4m',2\xi'+j)\right]_{i,j=0}^1 \\
	= -\frac{\chi_{m>m'}}{2\pi\i}\int_{\gamma_{0,1}} \Phi(z)^{m-m'}z^{\xi'-\xi}\frac{\d z}{z} \\
	+ \frac{1}{(2\pi\i)^2}\int_{\gamma_1}\int_{\gamma_{0,1}} \frac{w^{\xi'}}{z^{\xi+1}}\frac{\rho_1(w)^{\frac{N}{2}-m'}}{z-w}\frac{(1-z^{-1})^{N}}{(1-w^{-1})^{N}} \\
	\times E(w)
	\begin{pmatrix}
	1 & 0 \\
	0 & 0
	\end{pmatrix}
	E(w)^{-1}\Phi(z)^{m-\frac{N}{2}}\d z\d w, \\
	-N \leq \xi,\xi' \leq -1, \quad 0<m,m'<N.
	\end{multline}
	Here 
	\begin{equation}
	\Phi(z) = \frac{1}{(1-z^{-1})^2}
	\begin{pmatrix}
	1 & \alpha^2z^{-1} \\
	\beta^2 & 1
	\end{pmatrix}^2
	\begin{pmatrix}
	1 & z^{-1} \\
	1 & 1
	\end{pmatrix}^2,
	\end{equation}
	$ \rho_1(z) $ and $ \rho_2(z) $ are the eigenvalues of $ \Phi(z) $,
	\begin{equation}
	\Phi(z) = E(z)
	\begin{pmatrix}
	\rho_1(z) & 0 \\
	0 & \rho_2(z)
	\end{pmatrix}
	E(z)^{-1},
	\end{equation}
	and $ \gamma_1 $ is a contour around $ 1 $ and $ \gamma_{0,1} $ is a contour around $ 0 $ and $ \gamma_1 $.
\end{theorem}

\begin{proof}
	As in the uniform case \eqref{eq:section_aztec_diamond:2weight} does not fit into the framework of Section \ref{sec:infinite} directly, since it is singular and has poles on the unit circle. Therefore we introduce an extra parameter $ 0<a<1 $ in the model, which we later take to $ 1 $. Let
	\begin{equation}
	\phi_{a,4k+1}(z) = 
	\begin{pmatrix}
	1 & \alpha^2a^{-1} z^{-1}\\
	\beta^2a^{-1} & 1
	\end{pmatrix},
	\end{equation}
	\begin{equation}
	\phi_{a,4k+2}(z) = \frac{1}{1-a^2z^{-1}}
	\begin{pmatrix}
	1 & \alpha^2a z^{-1}\\
	\beta^2a & 1
	\end{pmatrix},
	\end{equation}
	\begin{equation}
	\phi_{a,4k+3}(z) = 
	\begin{pmatrix}
	1 & a^{-1} z^{-1}\\
	a^{-1} & 1
	\end{pmatrix},
	\end{equation}
	\begin{equation}
	\phi_{a,4k+4}(z) = \frac{1}{1-a^2z^{-1}}
	\begin{pmatrix}
	1 & a z^{-1}\\
	a & 1
	\end{pmatrix},
	\end{equation}
	for $ k = 0,1,\cdots, N-1 $. Consider the model defined by \eqref{eq:productofdeterminants}, \eqref{eq:endpoints} with transition matrices $ T_m=T_{\phi_{a,m}} $ and $ M = -N $. To use Theorem \ref{thm:main} we need a factorization $ \phi_a=\phi_{a,+}\phi_{a,-}=\widetilde{\phi}_{a,-}\widetilde{\phi}_{a,+} $ of the matrix
	\begin{equation}
	\phi_a(z) = \Phi_a(z)^N,
	\end{equation}
	where
	\begin{equation}
	\Phi_a(z) = \phi_{a,1}(z)\phi_{a,2}(z)\phi_{a,3}(z)\phi_{a,4}(z).
	\end{equation}
	The way we introduce the $ a $ parameter in the model is done so that the zeros and poles of $ \det \phi_a $ are away from the unit circle and so that the winding number of $ \det \phi_a $ is $ -2N $. The conditions in Theorem \ref{thm:p-periodic_admissible} are therefore fulfilled. As in the $ p=1 $ case, this is an interesting model on the Aztec diamond by itself. However,  the formula we obtain for general $ a\in (0,1) $ seems to complicated to use for asymptotic analysis.  
	
	By the procedure discussed in Section \ref{sec:method_factorization} we use \eqref{eq:2p_switching_rule} to obtain 
$$
	\phi_a(z) = \prod_{k=1,\text{even}}^{4N}\phi'_{a,k}(z)\prod_{k=1,\text{odd}}^{4N}\phi'_{a,k}(z) 
	= \prod_{k=1,\text{odd}}^{4N}\phi''_{a,k}(z)\prod_{k=1,\text{even}}^{4N}\phi''_{a,k}(z),
$$
	where, for each $ k $, $ \phi_{a,k}' $, $ \phi_{a,k}'' $ and $ \phi_{a,k} $ are of the same type and have the same determinant. So $ \phi'_{a,k} $ and $ \phi''_{a,k} $ does not have poles or singularities inside the unit circle if $ k $ is odd, and $ \phi'_{a,k} $ and $ \phi''_{a,k} $ does not have poles or singularities outside the unit circle if $ k $ is even, except possible at zero and infinity. To obtain the right behavior at zero and infinity we compensate with a factor $ z^N $. More precisely, the factors in the two factorizations of $ \phi_a $ become
	\begin{equation}
	\phi_{a,+}(z) = z^N\prod_{k=1,\text{odd}}^{4N}\phi''_{a,k}(z)C, \quad  \phi_{a,-}(z) = C^{-1}z^{-N}\prod_{k=1,\text{even}}^{4N}\phi''_{a,k}(z),
	\end{equation}
	and
	\begin{equation}
	\widetilde \phi_{a,+}(z) = \widetilde Cz^N\prod_{k=1,\text{odd}}^{4N}\phi'_{a,k}(z), \quad  \widetilde \phi_{a,-}(z) = z^{-N}\prod_{k=1,\text{even}}^{4N}\phi'_{a,k}(z)\widetilde C^{-1},
	\end{equation}
	where $ C $ and $ \widetilde C $ are constants to normalize the leading coefficient at infinity. Theorem \ref{thm:main} then applies. 
	
	Before applying Theorem \ref{thm:main}, we make an important observation that will be used later. The $ \phi'_{a,m} $ are created by applying the rule \eqref{eq:2p_switching_rule} many times, in a particular order (explained in Section \ref{sec:method_factorization}). However, if we use this rule in the case $ a=1$, it does not effect the matrices. For example
	\begin{equation}
	\begin{pmatrix}
	1 & \alpha^2az^{-1} \\
	\beta^2 a & 1
	\end{pmatrix}
	\begin{pmatrix}
	1 & a^{-1}z^{-1}\\
	a^{-1} & 1 
	\end{pmatrix}
	=
	\begin{pmatrix}
	1 & a^{-1}xz^{-1}\\
	a^{-1}x^{-1} & 1 
	\end{pmatrix}
	\begin{pmatrix}
	1 & \beta^2axz^{-1}\\
	\alpha^2ax^{-1} & 1 
	\end{pmatrix},
	\end{equation}
	where $ x = \frac{\alpha^2a+a^{-1}}{\beta^2a+a^{-1}} $ and in the the case $ a=1 $ this equality becomes
	\begin{equation}
	\begin{pmatrix}
	1 & \alpha^2z^{-1} \\
	\beta^2 & 1
	\end{pmatrix}
	\begin{pmatrix}
	1 & z^{-1}\\
	1 & 1 
	\end{pmatrix}
	=
	\begin{pmatrix}
	1 & \alpha^2z^{-1} \\
	\beta^2 & 1
	\end{pmatrix}
	\begin{pmatrix}
	1 & z^{-1}\\
	1 & 1 
	\end{pmatrix}.
	\end{equation}
	So the factors $ \phi'_{1,k} $, $ k=1,\cdots,4N $ are not complicated and 
	\begin{equation}
	\prod_{k=1,\text{even}}^{4N}\phi'_{1,k}(z) = (1-z^{-1})^{-N}\Phi(z)^\frac{N}{2}.
	\end{equation}
	Moreover, this switching rule is continuous with respect to the parameter $ a $. This tells us that even if $ \widetilde \phi_{a,-} $ is complicated for $ a \in (0,1) $, it simplifies significantly when $ a \to 1 $ (the same is true for $ \widetilde \phi_{a,+} $), namely
	\begin{equation}\label{eq:section_aztec_diamond:limit_factor}
	\prod_{k=1,\text{even}}^{4N}\phi'_{a,k}(z) \to \prod_{k=1,\text{even}}^{4N}\phi'_{1,k}(z) = (1-z^{-1})^{-N}\Phi(z)^\frac{N}{2},
	\end{equation}
	as $ a \to 1 $.
	
	Now, by Theorem \ref{thm:main} we obtain, if $ -N \leq \xi' $,
	\begin{multline}\label{eq:section_aztec_diamond:kernel_with_a}
	\left[K_{top}^{(a)}(4m,2\xi+i;4m',2\xi'+j)\right]_{i,j=0}^1 \\
	= -\frac{\chi_{m>m'}}{2\pi\i}\int_{\gamma_{0,1}} \Phi_a(z)^{m-m'}z^{\xi'-\xi}\frac{\d z}{z} \\
	+ \frac{1}{(2\pi\i)^2}\int_{\gamma_a}\int_{\gamma_{0,1,a}} \frac{w^{\xi'}}{z^{\xi+1}}\frac{1}{z-w} \Phi_a(w)^{N-m'} \left(\prod_{k=1,\text{odd}}^{4N}\phi'_{a,k}(w)\right)^{-1} \\
	\times\left(\prod_{k=1,\text{even}}^{4N}\phi'_{a,k}(z)\right)^{-1}\Phi_a(z)^m\d z\d w,
	\end{multline}
	where $ \gamma_a $ is a simple curve with $ a^2 $ in the interior and $ a^{-2} $ in the exterior, and $ \gamma_{0,1,a} $ is a simple curve with $ 0 $, $ 1 $ and $ \gamma_a $ in the interior. We would like to take $ a \to 1 $. The problem however is that the integrand with respect to $ w $ is singular both at $ a^2 $ and $ a^{-2} $ which lie on different sides of $ \gamma_a $, see Figure \ref{fig:section_aztec_diamond:contours}, which complicates the limit procedure. To solve this problem we go to the eigenvalues.
	
	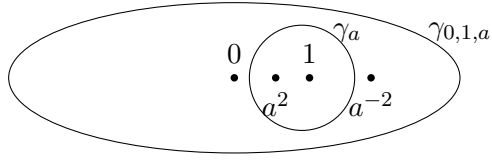
\begin{figure}[t]
		\begin{center}
			\begin{tikzpicture}[scale=1]
			\draw (.9,0) circle (.7);
			\draw (.55,0) node[circle,fill,inner sep=1pt,label=below:$a^2$]{};
			\draw (1,0) node[circle,fill,inner sep=1pt,label=above:$1$]{};
			\draw (1.818,0) node[circle,fill,inner sep=1pt,label=below:$a^{-2}$]{};
			\draw (0,0) node[circle,fill,inner sep=1pt,label=above:$0$]{};
			\draw (0,0) ellipse (3 and 1);
			\draw (1.5,.6) node{$ \gamma_a$};
			\draw (3,.6) node{$ \gamma_{0,1,a}$};
			\end{tikzpicture}
		\end{center}
		\caption{The contours of integration. The point $ a^2 $ is a pole for $ \rho_{a,1} $ and $ a^{-2} $ is a zero for $ \rho_{a,2} $.
			\label{fig:section_aztec_diamond:contours}}
	\end{figure}
	
	We do not need explicit formulas for the eigenvalues, but only the behavior of the eigenvalues near $z=1$. Nevertheless, for the sake of clarity we calculate them explicitly. First, observe that 
	\begin{equation}
	\det \Phi_a(z) = \frac{(1-a^{-2}z^{-1})^2}{(1-a^2z^{-1})^2},
	\end{equation}
	and
	\begin{equation}
	\Tr \Phi_a(z) = \frac{1}{(1-a^2z^{-1})^2}\left(2(1+z^{-1})^2+z^{-1}(a+a^{-1})^2(\alpha^2+\beta^2)\right).
	\end{equation}
The eigenvalues of $ \Phi_a $ become therefore	\begin{multline}
	\rho_{a,1}(z) = \frac{1}{(z-a^2)^2}\left((z+1)^2+\frac{1}{2}z(a+a^{-1})^2(\alpha^2+\beta^2)\right. \\
	\left.+ (a+a^{-1})(\alpha+\beta)\sqrt{z(z^2+xz+1)}\right)
	\end{multline}
	and
	\begin{multline}
	\rho_{a,2}(z) = \frac{1}{(z-a^2)^2} \left((z+1)^2+\frac{1}{2}z(a+a^{-1})^2(\alpha^2+\beta^2) \right. \\
	\left.- (a+a^{-1})(\alpha+\beta)\sqrt{z(z^2+xz+1)}\right)
	\end{multline}
	where
	\begin{equation}
	x = \frac{1}{4}\left((a+a^{-1})^2(\alpha^2+\beta^2)-2(a-a^{-1})^2\right)\geq 2.
	\end{equation}
	That $ x \geq 2 $ tells us that the branch of $ \sqrt{z(z^2+xz+1)} $ can be taken with the cuts 
	\begin{equation}  
	\left(-\infty,-\frac{x}{2}-\sqrt{\frac{x^2}{4}-1}\right] \text{ and } \left[-\frac{x}{2}+\sqrt{\frac{x^2}{4}-1},0\right].
	\end{equation}
	Take $ \sqrt{z(z^2+xz+1)} $ to be positive when $ z>0$ which, together with $ \det \Phi_a = \rho_{a,1}\rho_{a,2} $, implies that $ \rho_{a,1} $ and $ \rho_{a,2} $ are analytic and non-zero for $ z $ away from the cuts, except for $ z = a^2 $ where $ \rho_{a,1} $ has a pole, and for $ z = a^{-2} $ where $ \rho_{a,2} $ is zero. 
	
	Let $ \Phi_a(w) = E_a(w)\Lambda_a(w)E_a(w)^{-1} $ be an eigenvalue decomposition of $ \Phi_a $. Use the eigenvalue decomposition to write $ \Phi_a^{N-m'} $ as the sum
	\begin{multline}
	\Phi_a(w)^{N-m'} = \rho_{a,1}(w)^{N-m'}E_a(w)
	\begin{pmatrix}
	1 & 0 \\
	0 & 0
	\end{pmatrix}
	E_a(w)^{-1} \\
	+ \rho_{a,2}(w)^{N-m'}E_a(w)
	\begin{pmatrix}
	0 & 0 \\
	0 & 1
	\end{pmatrix}
	E_a(w)^{-1},
	\end{multline}
	and insert it in \eqref{eq:section_aztec_diamond:kernel_with_a}. Then the term with $ \rho_{a,2} $ is zero, since the integrand depending on $ w $ is analytic inside $ \gamma_a $. For the other term, recall that $ \phi'_{a,m} $ is analytic outside the unit circle if $ m $ is even but is singular at $ a^{-2} $ if $ m $ is odd. Moreover, 
	\begin{multline}
	\rho_{a,1}(w)^{N-m'}E_a(w)
	\begin{pmatrix}
	1 & 0 \\
	0 & 0
	\end{pmatrix}
	E_a(w)^{-1}
	\left(\prod_{k=1,\text{odd}}^{4N}\phi'_{a,k}(w)\right)^{-1} \\
	= \rho_{a,1}(w)^{-m'}E_a(w)
	\begin{pmatrix}
	1 & 0 \\
	0 & 0
	\end{pmatrix}
	E_a(w)^{-1}
\Phi_a(w)^N	\left(\prod_{k=1,\text{odd}}^{4N}\phi'_{a,k}(w)\right)^{-1} \\
	=\rho_{a,1}(w)^{-m'}E_a(w)
	\begin{pmatrix}
	1 & 0 \\
	0 & 0
	\end{pmatrix}
	E_a(w)^{-1}
	\prod_{k=1,\text{even}}^{4N}\phi'_{a,k}(w),
	\end{multline}
	which is analytic at $w= a^{-2} $. Move the contour $ \gamma_a $ to a contour $ \gamma_1 $ containing $ a^2 $, $1$ and $ a^{-2}$ in its interior. By rewriting \eqref{eq:section_aztec_diamond:kernel_with_a} in this way we may, by analyticity of the kernel, take the limit $ a \to 1 $. This shows that 
	\begin{multline}
	\lim_{a\to1} \left[K_{top}^{(a)}(4m,2\xi+i;4m',2\xi'+j)\right]_{i,j=0}^1 \\
	= -\frac{\chi_{m>m'}}{2\pi\i}\int_{\gamma_{0,1}} \Phi(z)^{m-m'}z^{\xi'-\xi}\frac{\d z}{z} \\
	+ \frac{1}{(2\pi\i)^2}\int_{\gamma_1}\int_{\gamma_{0,1}} \frac{w^{\xi'}}{z^{\xi+1}}\frac{\rho_1(w)^{-m'}}{z-w} E(w)
	\begin{pmatrix}
	1 & 0 \\
	0 & 0
	\end{pmatrix}
	E(w)^{-1}
	\prod_{k=1,\text{even}}^{4N}\phi'_{1,k}(w) \\
	\times\left(\prod_{k=1,\text{even}}^{4N}\phi'_{1,k}(z)\right)^{-1}\Phi(z)^m\d z\d w.
	\end{multline}
 After 	combining this with  \eqref{eq:section_aztec_diamond:limit_factor}, we obtain the statement.
\end{proof}

\subsection{Domino tiling of the Aztec diamond $ p=3 $}\label{sec:aztec_diamond:p=3}

The last example on the Aztec diamond that we discuss is  a model with a higher periodicity. Numerical and preliminary computations suggest that this model has two distinct gas phases, see Figure \ref{fig:section_aztec_diamond:plot3x2}.

\begin{figure}[t]
	\begin{center}
		\includegraphics[scale=.45]{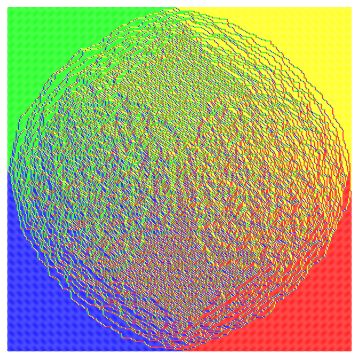}
		\hspace{.5cm}
		\includegraphics[scale=.45]{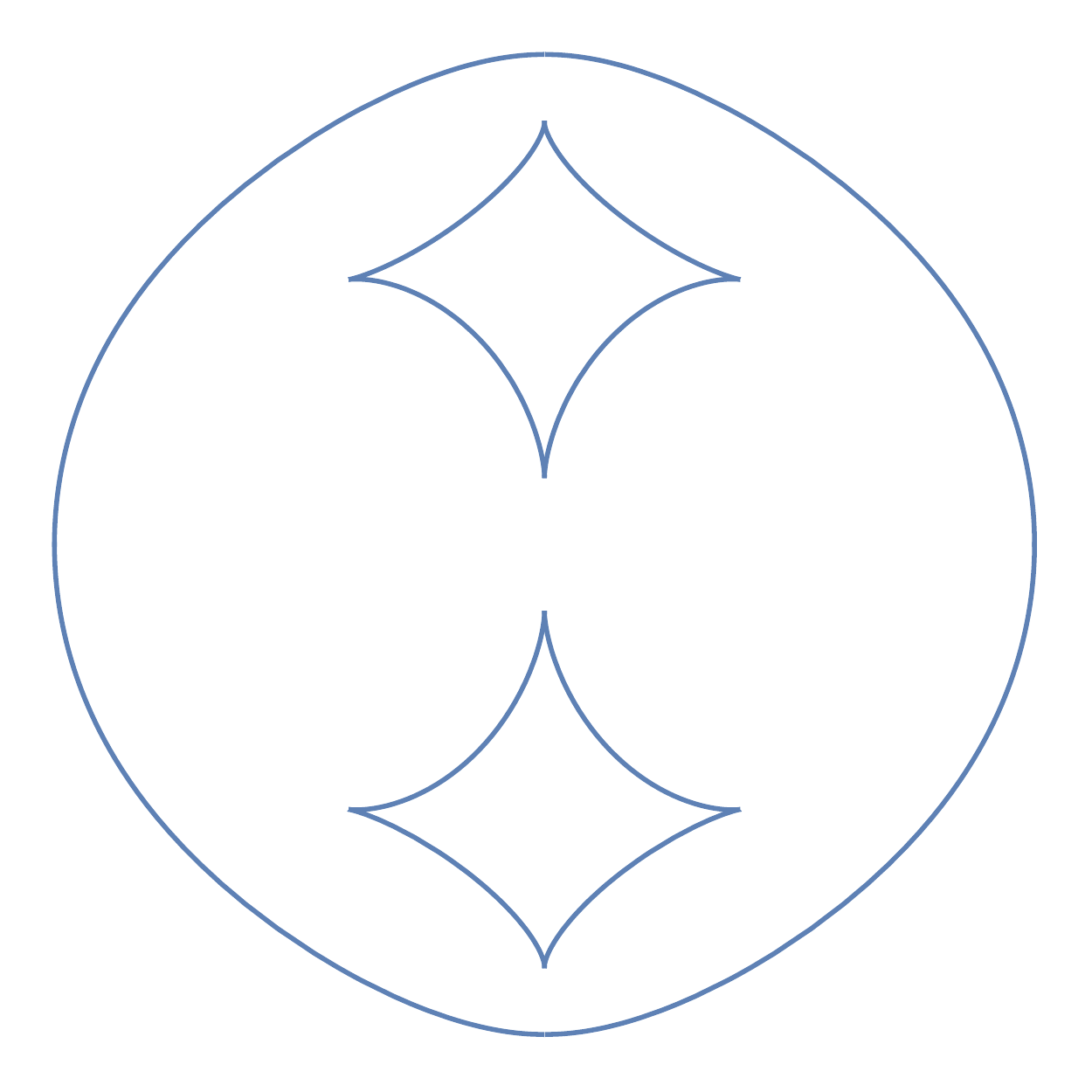}
	\end{center}
	\caption{Numerical and formal computations, here with $ \alpha_0 = 0.2 $, $ \alpha_1 = 0.7 $, $ \alpha_2=\alpha_0^{-1}\alpha_1^{-1} $ and $ \beta_0=\beta_1=\beta_2=1 $, indicate that there are two distinct gas phases. \label{fig:section_aztec_diamond:plot3x2}}
\end{figure}

Let $ \alpha_0,\alpha_1,\alpha_2,\beta_0,\beta_1,\beta_2 $ be positive parameters such that $ \alpha_0\alpha_1\alpha_2 = \beta_0\beta_1\beta_2 = 1 $. Consider the $ 6 N $-sized Aztec diamond and set $ b_{m,x} = 1 $ for all $ m $ and $ x $, $ c_{m,x} = d_{m+1,x} = \alpha_i $ if $ x\equiv i \mod 3 $ and $ m $ is even, $ c_{m,x} = d_{m+1,x} = \beta_i $ if $ x\equiv i \mod 3 $ and $ m $ is odd. That is, consider the probability measure defined by \eqref{eq:productofdeterminants}, \eqref{eq:endpoints} with transition matrices $ T_m = T_{\phi_m} $ with symbols
\begin{equation}
\phi_{4k+1}(z) = M_d(z;\alpha_0,\alpha_1,\alpha_2),
\end{equation}
\begin{equation}
\phi_{4k+2}(z) = N_d(z;\alpha_0,\alpha_1,\alpha_2),
\end{equation}
\begin{equation}
\phi_{4k+3}(z) = M_d(z;\beta_0,\beta_1,\beta_2),
\end{equation}
and
\begin{equation}
\phi_{4k+4}(z) = N_d(z;\beta_0,\beta_1,\beta_2),
\end{equation}
for $ k=0,\cdots,3N-1 $ and $ M=-2N $. Here we use the notation
\begin{equation}
M_d(z;\alpha_0,\alpha_1,\alpha_2) = M(1/z;(\alpha_0,\alpha_1,\alpha_2))^T,
\end{equation}
and
\begin{equation}
N_d(z;\alpha_0,\alpha_1,\alpha_2) = N(1/z;(\alpha_0,\alpha_1,\alpha_2))^T,
\end{equation}
recall \eqref{eq:defMa} and \eqref{eq:defNa}. Then
\begin{equation}\label{eq:section_aztec_diamond:3x2weight}
\phi(z) = \prod_{m=1}^{12N}\phi_m(z) = \Phi(z)^{3N},
\end{equation}
where
\begin{equation}
\Phi(z) = \phi_1(z)\phi_2(z)\phi_3(z)\phi_4(z).
\end{equation}

\begin{theorem}\label{thm:section_aztec_diamond:3x2}
	Consider the $ 3\times 2 $-periodic Aztec diamond defined above. The kernel for the corresponding point process is given by
	\begin{multline}
	\left[ K(12m, 3\xi +j; 12m', 3 \xi'+i) \right]_{i,j=0}^2 \\
	= - \frac{\chi_{m > m'}}{2\pi \i} \oint_{\gamma_\text{int}}
	A(z)^{m-m'}B(z)^{m-m'} z^{\xi'-\xi} \frac{\d z}{z}
	\\
	+  \frac{1}{(2\pi \i)^2}
	\int_{\gamma_\text{int}}\int_{\gamma_\text{ext}} 
	A(w)^{N-m'}B(w)^{-m'} A(z)^{m-N} B(z)^m  \frac{w^{\xi'}}{z^{\xi+1}} \frac{\d z \d w}{z-w}, \\
	\xi,\xi' \in \mathbb{Z}, \quad 0 < m,m'<N,
	\end{multline}
	where
	\begin{multline}
	A(z) =  N_d(z;\alpha_0,\alpha_1\alpha_2) N_d(z;\beta_2c_0,\beta_0c_1,\beta_1c_2) \\
	\times N_d(z;\alpha_2c_0,\alpha_0c_1,\alpha_1c_2)  N_d(z;\beta_1c_2c_0,\beta_2c_0c_1,\beta_0c_1c_0) \\
	\times N_d(z;\alpha_1c_2c_0,\alpha_2c_0c_1,\alpha_0c_1c_0) N_d(z;\beta_0,\beta_1,\beta_2),
	\end{multline}
	\begin{multline}
	B(z) =  M_d(z;\alpha_0,\alpha_1,\alpha_2)  M_d(z;\beta_1c_2c_0,\beta_2c_0c_1,\beta_0c_1c_0) \\
	\times M_d(z;\alpha_1c_2c_0,\alpha_2c_0c_1,\alpha_0c_1c_0) M_d(z;\beta_2c_0,\beta_0c_1,\beta_1c_2) \\
	\times M_d(z;\alpha_2c_0,\alpha_0c_1,\alpha_1c_2) M_d(z;\beta_0,\beta_1,\beta_2),
	\end{multline}
	and
	\begin{equation}
	c_0=\frac{\alpha_0+\beta_0}{\alpha_2+\beta_2}, \quad c_1=\frac{\alpha_1+\beta_1}{\alpha_0+\beta_0}, \quad c_2=\frac{\alpha_2+\beta_2}{\alpha_1+\beta_1},
	\end{equation}
	and $ \gamma_\text{int} $ is a contour around one and zero with $ - 1 $ in the exterior, and $ \gamma_\text{ext} $ is a contour around $ \gamma_\text{int} $, with $ -1 $ in the exterior.
\end{theorem}

Observe that, since $ A $ and $ B $ are explicit matrices, an asymptotic analysis of the correlation kernel  by a steepest descent analysis should be within reach.

To prove this theorem, we need, as in the uniform Aztec diamond and the two-periodic Aztec diamond, to introduce a parameter $ a \in (0,1) $, which we later take to one. In contrast to the proof of Theorem \ref{thm:two_periodic_aztec_diamond}, it will not be necessary to go to the eigenvalues,.

\begin{proof}
	As in the other two examples \eqref{eq:section_aztec_diamond:3x2weight} does not fit into the framework of Section \ref{sec:infinite}, since it is singular and has poles on the unit circle. As before, we therefore introduce an extra parameter $ 0<a<1 $ in the model and take the limit $ a\to 1 $. Let
	\begin{equation}
	\phi_{a,4k+1}(z) =  M_d(z;\alpha_0a^{-1},\alpha_1a^{-1},\alpha_2a^{-1}),
	\end{equation}
	\begin{equation}
	\phi_{a,4k+2}(z) =  N_d(z;\alpha_0a,\alpha_1a,\alpha_2a),
	\end{equation}
	\begin{equation}
	\phi_{a,4k+3}(z) =  M_d(z;\beta_0a^{-1},\beta_1a^{-1},\beta_2a^{-1}),
	\end{equation}
	and
	\begin{equation}
	\phi_{a,4k+4}(z) =  N_d(z;\beta_0a,\beta_1a,\beta_2a),
	\end{equation}
	for $ k=0,\cdots,3N-1 $. Consider the model defined by \eqref{eq:productofdeterminants} and \eqref{eq:endpoints} with transition matrices $ T_m=T_{\phi_{a,m}} $ and $ M = -2N $. The parameter $ a $ is introduced so that the determinant of
	\begin{equation}
	\phi_a(z) = \prod_{m=1}^{12N}\phi_{m,a}(z),
	\end{equation}
	does not have zeros and poles on the unit circle and has winding number $-6N$. By Theorem \ref{thm:p-periodic_admissible} $ \phi_a $ has a factorization  \eqref{eq:WHfact}. As before it is the factorization $ \phi_a = \widetilde \phi_{a,-}\widetilde \phi_{a,+} $ which gives us the right kernel in Theorem \ref{thm:main}. This factorization can be constructed by switching the factors in $ \phi_a $ pairwise, as explained in Section \ref{sec:method_factorization}. When $ a \in (0,1) $ the factorization seems too complicated to use for an asymptotic analysis. However, as in the two-periodic case, the factorization simplifies when $ a=1 $. When $ a = 1 $ the switching is not as straightforward as it was in the two-periodic case, but at least we are in the situation with \eqref{eq:finite_steps} and hence we obtain an explicit factorization. We prove this before applying Theorem \ref{thm:main}. 
	
	Explicitly, in this situation, Proposition \ref{lem:section_matrix_case:interchange} becomes
	\begin{multline}
	M_d(z;a_0,a_1,a_2)  N_d(z;b_0,b_1,b_2) \\
	=  N_d(z;b_2x_0,b_0 x_1,b_1 x_2) M_d(z;a_2x_0,a_0x_1,a_1 x_2),
	\end{multline}
	where
	\begin{equation}
	x_0=\frac{a_0+b_0}{a_2+b_2}, \quad x_1=\frac{a_1+b_1}{a_0+b_0} \quad \text{and} \quad x_2=\frac{a_2+b_2}{a_1+b_1}.
	\end{equation}
	Apply this equality to $ \phi $ pairwise six times. Each time we obtain a factor of $ N_d $ to the left, and a factor $ M_d $ to the right, namely
	\begin{multline}
	\Phi(z)^{3N} = \\
	\left( M_d(z;\alpha_0,\alpha_1,\alpha_2) N_d(z;\alpha_0,\alpha_1,\alpha_2) M_d(z;\beta_0,\beta_1,\beta_2) N_d(z;\beta_0,\beta_1,\beta_2)\right)^{3N} \\
	=  N_d(z;\alpha_0,\alpha_1,\alpha_2) N_d(z;\beta_2c_0,\beta_0c_1,\beta_1c_2) N_d(z;\alpha_2c_0,\alpha_0c_1,\alpha_1c_2) \\
	\times  N_d(z;\beta_1c_2c_0,\beta_2c_0c_1,\beta_0c_1c_0) N_d(z;\alpha_1c_2c_0,\alpha_2c_0c_1,\alpha_0c_1c_0) N_d(z;\beta_0,\beta_1,\beta_2) \\
	\times\left( M_d(z;\alpha_0,\alpha_1,\alpha_2) N_d(z;\alpha_0,\alpha_1,\alpha_2) M_d(z;\beta_0,\beta_1,\beta_2) N_d(z;\beta_0,\beta_1,\beta_2)\right)^{3(N-1)}\\
	\times M_d(z;\alpha_0,\alpha_1,\alpha_2)  M_d(z;\beta_1c_2c_0,\beta_2c_0c_1,\beta_0c_1c_0) M_d(z;\alpha_1c_2c_0,\alpha_2c_0c_1,\alpha_0c_1c_0) \\
	\times  M_d(z;\beta_2c_0,\beta_0c_1,\beta_1c_2) M_d(z;\alpha_2c_0,\alpha_0c_1,\alpha_1c_2) M_d(z;\beta_0,\beta_1,\beta_2) \\
	= A(z)\Phi(z)^{3(N-1)}B(z).
	\end{multline}
	Repeat this for a total of $ N $ times to obtain 
	\begin{equation}
	\Phi(z)^{3N} = A(z)^NB(z)^N.
	\end{equation}
	This tells us that
	\begin{equation}
	\widetilde{\phi}_{1,+}(z) = z^{2N}CB(z)^N,
	\end{equation}
	and
	\begin{equation}
	\widetilde{\phi}_{1,-}(z) = z^{-2N}A(z)^NC^{-1},
	\end{equation}
	where $ z^{\pm 2N} $ and $ C $ are such that $\widetilde \phi_+$ and $\widetilde \phi_-$ have the correct behavior  at zero respectively infinity.
	
	Now, by Theorem \ref{thm:main} 
	\begin{multline}\label{eq:section_3x2_aztec_diamond:kernel}
	\left[ K_{top}(12m, 3\xi +j; 12m', 3 \xi'+i) \right]_{i,j=0}^2 \\
	= - \frac{\chi_{m > m'}}{2\pi \i} \oint_{\gamma_\text{int}}
	\left(\phi_{a,1}(z)\phi_{a,2}(z)\phi_{a,3}(z)\phi_{a,4}(z)\right)^{3(m-m')} z^{\xi'-\xi} \frac{\d z}{z}
	\\
	+  \frac{1}{(2\pi \i)^2}
	\int_{\gamma_\text{int}}\int_{\gamma_\text{ext}} 
	\left(\phi_{a,1}(w)\phi_{a,2}(w)\phi_{a,3}(w)\phi_{a,4}(w)\right)^{3(N-m')} \widetilde \phi_{a,+}(w)^{-1} \\
	\times\widetilde \phi_{a,-}(z)^{-1} \left(\phi_{a,1}(z)\phi_{a,2}(z)\phi_{a,3}(z)\phi_{a,4}(z)\right)^{3m}  \frac{w^{\xi'+2N}}{z^{\xi+2N+1}} \frac{\d z \d w}{z-w}.
	\end{multline}
	
	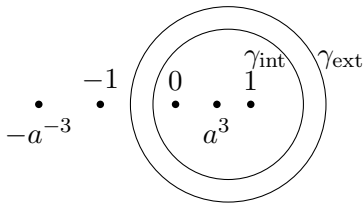
\begin{figure}[t]
		\begin{center}
			\begin{tikzpicture}[scale=1]
			\draw (.7,0) circle (1.3);
			\draw (.7,0) circle (1);
			\draw (.55,0) node[circle,fill,inner sep=1pt,label=below:$a^3$]{};
			\draw (1,0) node[circle,fill,inner sep=1pt,label=above:$1$]{};
			\draw (-1,0) node[circle,fill,inner sep=1pt,label=above:$-1$]{};
			\draw (-1.818,0) node[circle,fill,inner sep=1pt,label=below:$-a^{-3}$]{};
			\draw (0,0) node[circle,fill,inner sep=1pt,label=above:$0$]{};
			\draw (1.2,.6) node{$ \gamma_\text{int}$};
			\draw (2.2,.6) node{$ \gamma_\text{ext} $};
			\end{tikzpicture}
		\end{center}
		\caption{The contours of integration. The point $ a^3 $ is a pole to $ \phi_{a,m} $ if $ m $ is even and $ -a^{-3} $ is a singular point to $ \phi_{a,m} $ if $ m $ is odd. \label{fig:section_3x2_aztec_diamond:contours}}
	\end{figure}
	
	Since the procedure which is used to obtain $ \phi_{a,\pm} $ is continuous in the parameter $ a \in (0,1] $,  we see that $ \widetilde\phi_{a,-} \to \widetilde\phi_{1,-} $ and $ \widetilde \phi_{a,+} \to \widetilde \phi_{1,+} $ as $ a \to 1 $. Also $ \phi_{a,m} \to \phi_m $  as $ a \to 1 $. Moreover $ \phi_{a,m} $ and $ \widetilde \phi_{a,\pm} $ have poles and singularities away from $ \gamma_\text{int} $ and $ \gamma_\text{ext} $ (Figure \ref{fig:section_3x2_aztec_diamond:contours}). There is therefore no problem to take $ a \to 1 $ in \eqref{eq:section_3x2_aztec_diamond:kernel}. The result follows by taking the limit $ a \to 1 $ and using that $ \Phi(z)^{3k} = A(z)^kB(z)^k $ for all $ k $.
\end{proof}

\section{Example: Lozenge tilings of an infinite hexagon}\label{section:example_lozenge_tiling}

In the final section of this paper we turn to lozenge tilings of a hexagon.  In contrast to the Aztec diamond, lozenge tilings of the finite hexagon can not be represented as an ensemble with infinitely many paths. Instead,  we will take the limit as the vertical sides of the hexagon become infinitely large. This should be compared with the well-known fact that the lozenge tilings for the finite hexagon typically do not fall in the Schur class (which makes these models much harder to solve asymptotically \cite{BKMM,G}).

Consider a hexagon with corners at $(0,-\tfrac 12 )$, $(N-M,-\tfrac12)$, $(N,M-\tfrac 12)$, $(N,M+n-\tfrac 12)$ and $(M,M+n-\tfrac 12)$ and $(0,n-\tfrac 12)$. We tile the hexagon with the following type of lozenges 
\begin{center}
	\tikz[scale=.3]{\draw (0,0) \lozr;} \quad \tikz[scale=.3]{\draw (0,0) \lozu;}\quad
	and \ \tikz[scale=.3]{\draw (0,0) \lozd;},
\end{center}
see Figure \ref{fig:lozenge_tiling_hexagon}. We assign a weight to each tiling $\mathcal T$ as follows. Each lozenge will have an individual weight
\begin{center}
	$w\left(\tikz[scale=.3,baseline=(current bounding box.center)]{\draw (0,-.5) \lozr; \filldraw (0,-.5) circle(3pt); \draw (0,-.5) node[below] {\tiny{$(m-1,x-\frac12)$}};}
		\right)=a_{m,x},$
	 \quad  	$w\left(\tikz[scale=.3,baseline=(current bounding box.center)]{\draw (0,-.5) \lozu;  \filldraw (0,-.5) circle(3pt); \draw (0,-.5) node[below] {\tiny{$(m-1,x-\frac12)$}};}
	 \right)=b_{m,x}$ \quad
	and \ 	$w\left(\tikz[scale=.3,baseline=(current bounding box.center)]{\draw (0,-.5) \lozd;  \filldraw (0,-.5) circle(3pt); \draw (0,-.5) node[below] {\tiny{$(m-1,x-\frac12)$}};}
	\right)=c_{m,x},$
\end{center}
for some parameters $ a_{m,x},b_{m,x},c_{m,x} \in (0,\infty) $, if the bottom-left corner of the lozenge is at position $ (m,x-\frac{1}{2}) $. The  weight of a tiling $ w(\mathcal T) $ will then be defined as the product of the weights of all lozenges in the tiling. With the weights at hand, we define a probability measure in the usual way,
\begin{equation}\label{eq:section_lozenge_tiling:measure}
\mathbb P(\mathcal T) = \frac{w(\mathcal T)}{\sum_{\tilde{\mathcal T}}w(\tilde{\mathcal T})},
\end{equation}
where the sum is taken over all possible tilings.

	\begin{figure}[t]
	\begin{center}
		
		\begin{subfigure}[t]{0.45\textwidth}
			\includegraphics[scale=0.23,trim={15cm 0 13cm 0},clip]{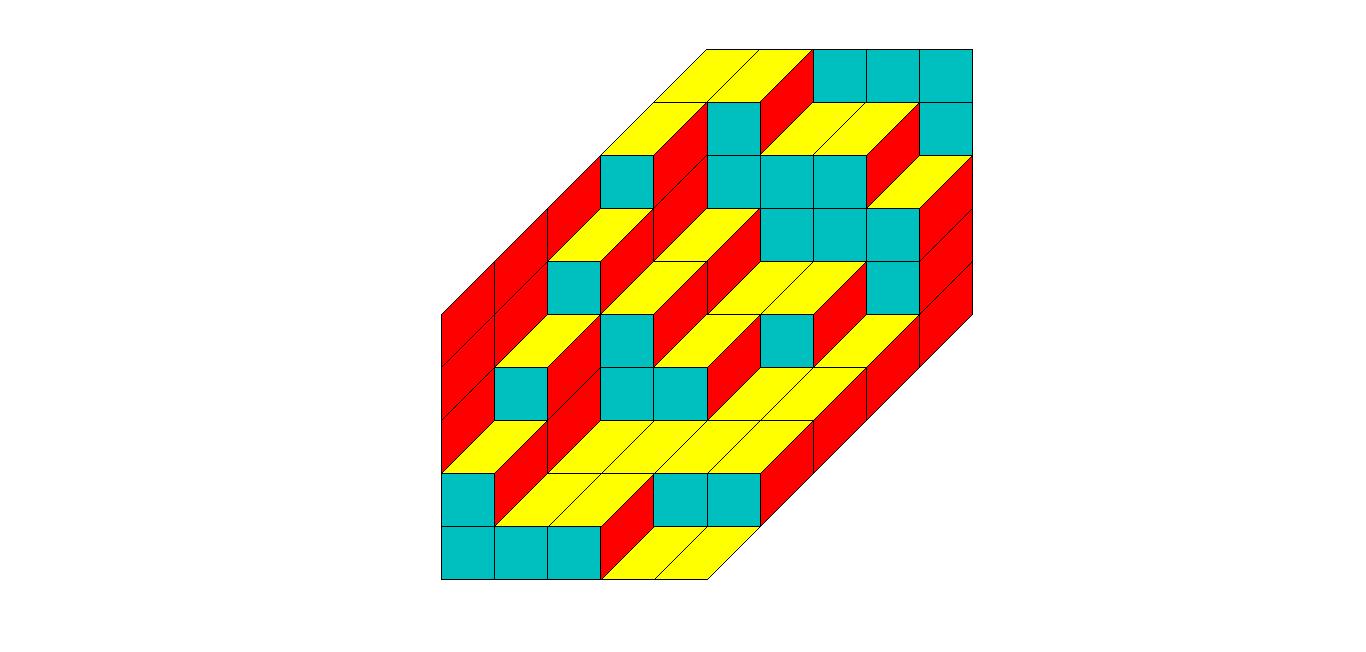}
		\end{subfigure}
		\hspace{0.5cm}
		\begin{subfigure}[t]{0.45\textwidth}
			\includegraphics[scale=0.23,trim={15cm 0 13cm 0},clip]{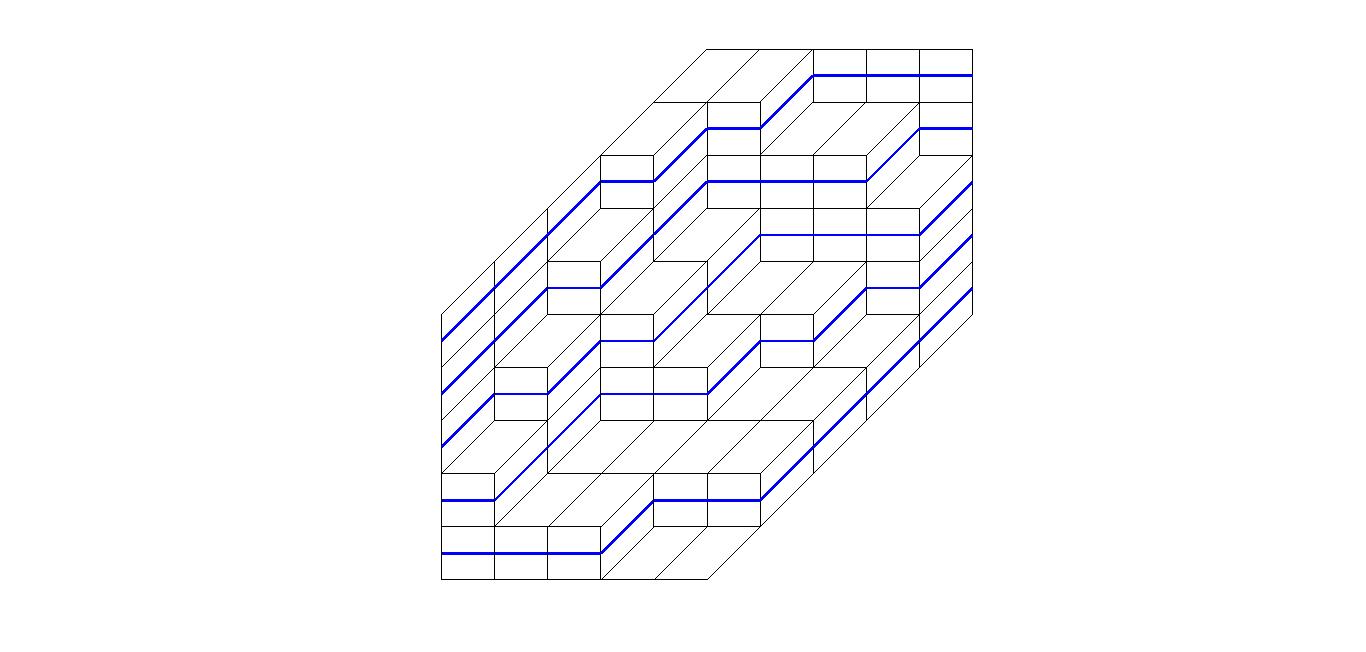}
		\end{subfigure}
		\caption{An example of a lozenge tiling of a $ 5\times 5 \times 5 $ sided hexagon. The first picture is thought of as boxes in a corner, while the other picture shows how it connect with non-intersecting paths. \label{fig:lozenge_tiling_hexagon}}
	\end{center}
\end{figure}

The weights that we put on the lozenges will be periodic.  That is, we assume that, for some $ p\geq 1 $, we have  $ a_{m,x+p}=a_{m,x} $, $ b_{m,x+p}=b_{m,x} $ and $ c_{m,x+p}=c_{m,x} $ for all $ m,x $. Assume then for simplicity that $ M $ and $ n $ are divisible by $ p $ and replace $ M $ by $ pM $ and $ n $ by $ pn $. By changing the weights on the tiles by $ a_{m,x} \mapsto \frac{a_{m,x}}{c_{m,x+1}} $ and $ b_{m,x} \mapsto \frac{b_{m,x}}{c_{m,x}} $, we may assume without loss of generality that $ c_{m,x} = 1 $ for all $ m,x $.

Each tiling give rise to a collection of non-intersecting paths as follows. Draw lines on two of the lozenges according to 
\begin{center}
	\tikz[scale=.3]{\draw (0,0) \lozr; \draw[very thick,blue] (0,.5)--(1,1.5);} \quad \tikz[scale=.3]{\draw (0,0) \lozu; \draw[very thick,blue] (0,.5)--(1,.5);}\quad
	and \ \tikz[scale=.3]{\draw (0,0) \lozd;}.
\end{center}
In this way a tiling corresponds to $ n $ non-intersecting paths on the directed graph in Figure \ref{fig:paths_lozenge_tiling_hexagon}, going from $ (0,0), (0,1),\cdots,(0,pn-1) $ to $ (N,pM), (N,pM+1),\cdots,(N,pM+pn-1) $. By this construction we obtain a natural weighting of the graph, inherited from the tiling. Set $ a_{m,x} $ on the edge going from $ (m-1,x) $ to $ (m,x+1) $ and set $ b_{m,x} $ on the edge going from $ (m-1,x) $ to $ (m,x) $. The Lindstr\"om-Gessel-Viennot Theorem then gives us, as discussed in Section \ref{sec:finite}, a probability measure on the space of all non-intersecting paths going from $ (0,0), (0,1),\cdots,(0,pn-1) $ to $ (N,pM), (N,pM+1),\cdots,(N,pM+pn-1) $. Moreover, if viewed as a measure on the tilings this  coincides with the measure \eqref{eq:section_lozenge_tiling:measure}. Hence \eqref{eq:section_lozenge_tiling:measure} is given by \eqref{eq:productofdeterminants} and \eqref{eq:endpoints} with transition matrix $ T_m $ with $ T_m(x,x) = b_{m,x} $, $ T_m(x,x+1) = a_{m,x} $ and $ T_m(x,y) = 0 $ otherwise for $ m=1,\cdots,N $. Each step corresponds to a Bernoulli step up. 

In our examples we will take $ n $, the number of paths, to infinity. As discussed in Section \ref{sec:infinite} the top paths and the bottom paths become two separated processes in the limit.

\begin{figure}[t]
\begin{center}
	\begin{tikzpicture}[scale=.5]
	\tikzset{->-/.style={decoration={
				markings,
				mark=at position .5 with {\arrow{stealth}}},postaction={decorate}}}
	\foreach \x in {0,1,...,5}
	{\draw[->-] (\x,7)--(\x+1,7);
		\foreach \y in {0,1,...,6}
		{\draw[->-] (\x,\y)--(\x+1,\y+1);
			\draw[->-] (\x,\y)--(\x+1,\y);
	}}
	
	\foreach \y in {1,2,3,4}
	{\draw (1,\y) node[circle,fill,inner sep=2pt]{};
		\draw (5,\y+2) node[circle,fill,inner sep=2pt]{};}
	\end{tikzpicture}
	\begin{tikzpicture}[scale=.5]
	\foreach \y in {0,1,...,6}
	{\draw (0,\y)--(6,\y);
		\foreach \x in {0,1,...,5}
		{\draw (\x,\y)--(\x+1,\y+1);
	}}
	{\draw (0,7)--(6,7);}
	\foreach \y in {1,2,3,4}
	{\draw (1,\y) node[circle,fill,inner sep=2pt]{};
		\draw (5,\y+2) node[circle,fill,inner sep=2pt]{};}
	\foreach \x/\y in {1/1, 1/2, 2/1, 2/4, 3/6, 4/4, 4/5, 4/6}
	{\draw[line width=1mm,blue] (\x,\y)--(\x+1,\y);
	}
	\foreach \x/\y in {1/3, 1/4, 2/2,2/5, 3/1, 3/3, 3/4, 4/2}
	{\draw[line width=1mm,blue] (\x,\y)--(\x+1,\y+1);
	}
	\end{tikzpicture}
	\caption{The directed graph on which the non-intersecting paths corresponding to a lozenge tiling of a hexagon are defined. \label{fig:paths_lozenge_tiling_hexagon}}
	\end{center}
\end{figure} 

\subsection{Lozenge tilings of the infinite hexagon $ p = 1 $}

If $ p = 1 $ the parameters do not depend on $ x $, $ b_{m,x} = b_m $ and $ a_{m,x} = a_m $. Each transition matrix $ T_m $ is then a scalar Toeplitz matrix with symbol
\begin{equation}
\phi_m(z) = a_m+b_mz.
\end{equation}

As discussed in Section \ref{sec:factorization_p1} the $ p=1 $ case is rather straightforward and will lead to the Schur measure that was introduced in \cite{O}.

	Take $N$ even and $M=N/2$. Take $ a_m = 1 $ and choose parameters $1<b_1,\ldots,b_{N/2}$ and $0<b_{N/2+1},\ldots,b_{N}<1.$
	Then the conditions in Theorem \ref{thm:main} are satisfied with 
	\begin{equation}
	\phi_+(z) = \prod_{k=\frac{N}{2}+1}^N(1+b_kz), \quad \phi_-(z) = \prod_{k=1}^{N/2}(1+b_kz),
	\end{equation}
	and we can take the limit $n \to \infty$. The point process on the line $m=N/2$ are equivalent to the Schur measure after applying the particle/hole duality (see the appendix of \cite{BOO} for more details on this principle). Indeed, the kernel for the bottom of the hexagon takes the form 
	\begin{equation}
	K(x,x')= -\frac{1}{(2\pi \i)^2}\iint_{|z|<|w|} \frac{\prod_{k=1}^\frac{N}{2}(1+b_k z) }{\prod_{k=1}^\frac{N}{2}(1+b_k w) } \frac{\prod_{k=\frac{N}{2}+1}^{N} (1+b_k w) }{\prod_{k=\frac{N}{2}+1}^{N}(1+b_{k} z) } \frac{w^{x'}}{z^{x+1}}\frac{\d z \d w}{z-w}.
	\end{equation}
	After deforming the contour for $z$ to be outside the contour for $w$, and changing $w$ and $z$ to $-w$ and $-z$, we find 
	\begin{equation}
	(-1)^{x'-x}K(x,y)=I-K_{Schur}\left(x+\frac{1}{2}-\frac{N}{2},x'+\frac{1}{2}-\frac{N}{2}\right)
	\end{equation}
	where $K_{Schur}$ is the kernel for the Schur measure as derived in \cite{O}.

Another (and perhaps more standard) way to introduce the Schur measure is by considering non-intersecting paths for which the first half transitions  are geometric jumps up and the second half are geometric jump down \cite{J02}. This in fact, leads to the point process that can be obtained from the above point process by putting a particle at each hole and removing the original particles. 
\subsection{Lozenge tilings of the infinite hexagon $ p = 2 $}

\begin{figure}[t]
	\begin{center}
		\includegraphics[scale=.50,trim={5cm 3.8cm 4cm 5cm},clip]{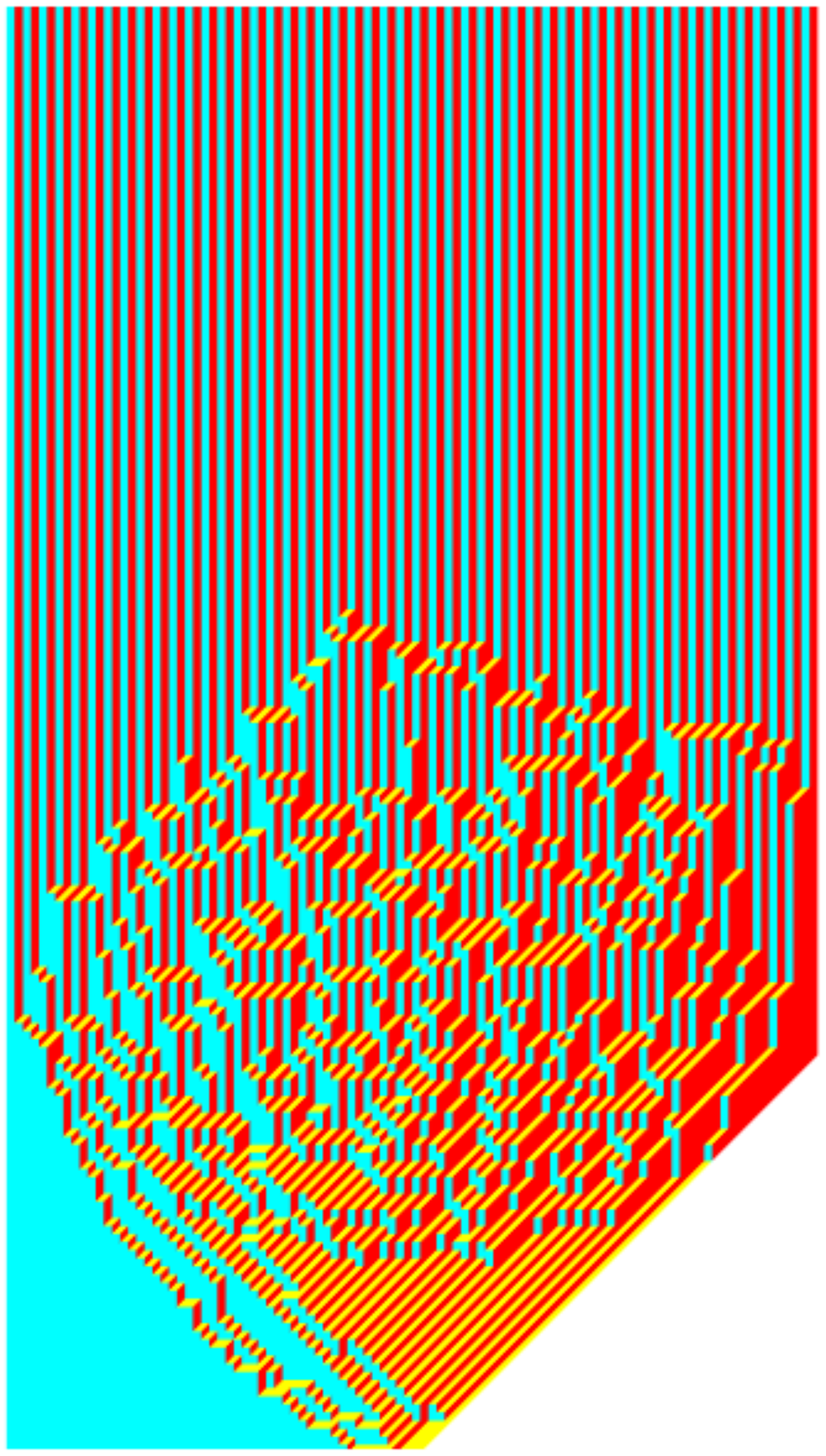}
	\end{center}
	\caption{A $ 2\times 2 $-periodic hexagon with parameters $ b=0.1 $, $c=4$, $\beta=10$, $ \gamma=0.25$ and $ a=d=\alpha=\delta=1 $. }
\end{figure}

If $ p=2 $, the values of $a_{m,x}$ and $b_{m,x}$ depend on whether $x$ is odd or even. Each transition matrix is  thus a block Toeplitz matrix, with a $ 2\times 2 $ matrix-valued symbol.  Here we will consider the case which is two-periodic also in the other direction, $ b_{m,x} = b_{m+2,x} $ and $ a_{m,x} = a_{m+2,x} $ for all $m$. Let $ a,b,c,d,\alpha,\beta,\gamma,\delta $ be positive numbers and consider the model defined by \eqref{eq:productofdeterminants}, \eqref{eq:endpoints} with $ M = N $ and with transition matrices $ T_m = T_{\phi_m} $ where the symbols $ \phi_m $ are given by
\begin{equation}
\phi_m(z) = 
\begin{pmatrix}
a & b \\
c z & d
\end{pmatrix},
\end{equation}
if $ m $ is odd and
\begin{equation}
\phi_m(z) = 
\begin{pmatrix}
\alpha & \beta \\
\gamma z & \delta
\end{pmatrix},
\end{equation}
if $ m $ is even, $ m=1,2,\cdots,4N $.  Then
\begin{equation}\label{eq:2x2-periodic_weight}
\phi(z) = \left(
\begin{pmatrix}
a & b \\
c z & b
\end{pmatrix}
\begin{pmatrix}
\alpha & \beta \\
\gamma z & \delta
\end{pmatrix}
\right)^{2N}.
\end{equation}
 A restriction we impose on the parameters is that
\begin{equation}
\frac{bc}{ad} \neq \frac{\beta\gamma}{\alpha\delta},
\end{equation}
and for simplicity we assume that $ \frac{\beta\gamma}{\alpha\delta} < 1 < \frac{bc}{ad} $. 

As $n \to \infty$ we obtain new determinantal point processes for which the correlation structure is described in the following theorem. This leads to a family of models that, to the best of our knowledge, have not been investigated in the literature before.
\begin{theorem}
	Consider the $ 2\times 2 $-periodic lozenge tilings of a hexagon defined above. As $ n \to \infty $ the bottom part of the hexagon converges to a determinantal point process with correlation kernel
	\begin{multline}
	\left[ K(4m, 2x+j; 4m', 2x'+i) \right]_{i,j=0}^{1} \\
	= - \frac{\chi_{m > m'}}{2\pi \i} \oint_{\gamma}A(z)^{m-m'}B(z)^{m-m'}z^{x'-x} \frac{\d z}{z}
	\\
	-  \frac{1}{(2\pi \i)^2}
	\int_{\gamma_{0,\frac{ad}{bc}}}\int_{\gamma_0} 
	A(w)^{N-m'}B(w)^{-m'}A(z)^{m-N}B(z)^m  \frac{w^{x'}}{z^{x+1}} \frac{\d z \d w}{z-w}, \\
	\qquad x,x' \in \mathbb Z, \, 0 < m, m' < N.
	\end{multline}
	where
	\begin{equation}
	A(z) = 
	\begin{pmatrix}
	x^\frac{1}{2} & 0 \\
	0 & x^{-\frac{1}{2}}
	\end{pmatrix}
	\begin{pmatrix}
	\alpha & \gamma \\
	\beta z &\delta
	\end{pmatrix}
	\begin{pmatrix}
	y^\frac{1}{2} & 0 \\
	0 & y^{-\frac{1}{2}}
	\end{pmatrix}
	\begin{pmatrix}
	\alpha & \beta \\
	\gamma z &\delta
	\end{pmatrix},
	\end{equation}
	\begin{equation}
	B(z) = 
	\begin{pmatrix}
	a & b \\
	c z & b 
	\end{pmatrix}
	\begin{pmatrix}
	y^{-\frac{1}{2}} & 0 \\
	0 & y^\frac{1}{2}
	\end{pmatrix}
	\begin{pmatrix}
	a & c \\
	b z & b 
	\end{pmatrix}
	\begin{pmatrix}
	x^{-\frac{1}{2}} & 0 \\
	0 & x^\frac{1}{2}
	\end{pmatrix},
	\end{equation}
	\begin{equation}
	x=\frac{a\beta + b \delta}{c \alpha+d\gamma}, \qquad y=\frac{d\gamma + c \delta}{b \alpha+ a\beta},
	\end{equation}
	and $ \gamma_0 $ is a curve around zero and $ \gamma_{0,\frac{ad}{bc}} $ is a curve around $ \gamma_0 $ and $ \frac{ad}{bc} $.
\end{theorem}

\begin{remark}
 The matrix-valued functions $ A $ and $ B $ commute, which  simplifies  an asymptotic analysis in the limit $ N\to \infty$, as it is possible to simultaneously diagonalize the factors in the integrand.
\end{remark}
\begin{remark}
	We could equally well consider the top part of the hexagon and obtain a limiting  process there. 
\end{remark}
\begin{remark}
	It is worth noting here that if $ M \neq N $, then the winding number of $ \det \phi $ is not equal to $ 2M $ and Theorem  \ref{thm:p-periodic_admissible} does not apply.
\end{remark}
The proof of this theorem is rather straightforward compared with the proofs in Section \ref{section:example_aztec_diamond}. It is not necessary to introduce an extra parameter, and we do not need to go to the eigenvalues.  

\begin{proof}
	Since the winding number of $ \det \phi $ is $ 2N $ and it does not have zeros or poles on the unit circle, Theorem \ref{thm:p-periodic_admissible} directly applies. To get an explicit formula in Theorem \ref{thm:main} we need to obtain the factorization $ \phi = \phi_+\phi_- $ (since we consider the bottom part) explicitly.
	
	The reason we can obtain an explicit formula is that $ \phi $ is of the form such that \eqref{eq:finite_steps} is true. To see this, note first that \eqref{eq:2p_switching_rule_bernoulli} can be written as	 
	\begin{equation}
	\begin{pmatrix}
	a & b \\
	c z & d 	 
	\end{pmatrix}
	\begin{pmatrix}
	\alpha & \beta \\
	\gamma z & \delta 	 
	\end{pmatrix}
	=
	\begin{pmatrix}
	x^\frac{1}{2} & 0 \\
	0 & x^{-\frac{1}{2}}
	\end{pmatrix}
	\begin{pmatrix}
	\alpha & \gamma \\
	\beta z & \delta
	\end{pmatrix}
	\begin{pmatrix}
	a & c \\
	b z & d	 
	\end{pmatrix}
	\begin{pmatrix}
	x^{-\frac{1}{2}} & 0 \\
	0 & x^\frac{1}{2}
	\end{pmatrix}.
	\end{equation}
	Apply this equality pairwise two times, with $ b $ and $ c $ interchanged and $ \gamma $ and $ \beta $ interchanged the second time, to obtain
	\begin{multline}
	\left(\begin{pmatrix}
	a & b \\
	c z & d 	 
	\end{pmatrix}
	\begin{pmatrix}
	\alpha & \beta \\
	\gamma z & \delta 	 
	\end{pmatrix}\right)^{2N}
	= \\
	\begin{pmatrix}
	x^\frac{1}{2} & 0 \\
	0 & x^{-\frac{1}{2}}
	\end{pmatrix}
	\begin{pmatrix}
	\alpha & \gamma \\
	\beta z & \delta
	\end{pmatrix}
	\begin{pmatrix}
	y^\frac{1}{2} & 0 \\
	0 & y^{-\frac{1}{2}}
	\end{pmatrix}
	\begin{pmatrix}
	\alpha & \beta \\
	\gamma z & \delta
	\end{pmatrix} \\
	\times\left(\begin{pmatrix}
	a & b \\
	c z & d 	 
	\end{pmatrix}
	\begin{pmatrix}
	\alpha & \beta \\
	\gamma z & \delta 	 
	\end{pmatrix}\right)^{2(N-1)} \\
	\times
	\begin{pmatrix}
	a & b \\
	c z & d	 
	\end{pmatrix}
	\begin{pmatrix}
	y^{-\frac{1}{2}} & 0 \\
	0 & y^\frac{1}{2}
	\end{pmatrix}
	\begin{pmatrix}
	a & c \\
	b z & d	 
	\end{pmatrix}
	\begin{pmatrix}
	x^{-\frac{1}{2}} & 0 \\
	0 & x^\frac{1}{2}
	\end{pmatrix}.
	\end{multline}
	Repeat this for a total of $ N $ times to obtain the factorization $ \phi(z) = \phi_+(z)\phi_-(z) $ with 
	\begin{equation}
	\phi_+(z) = A(z)^NC,
	\end{equation}
	and 
	\begin{equation}
	\phi_-(z) = C^{-1}B(z)^N,
	\end{equation}
	where $ C $ is a normalizing factor.
	
	Using the above factorization and that
	\begin{equation}
	\left(\begin{pmatrix}
	a & b \\
	c z & d 	 
	\end{pmatrix}
	\begin{pmatrix}
	\alpha & \beta \\
	\gamma z & \delta 	 
	\end{pmatrix}\right)^{2k}
	= A(z)^kB(z)^k,
	\end{equation}
	in Theorem \ref{thm:main} give the correlation kernel in the statement.
\end{proof}

\end{document}